\RequirePackage{fix-cm}
\documentclass[smallextended]{svjour3}       
\smartqed  
\usepackage{graphicx}
\usepackage{amsmath,amssymb,amsfonts}
\usepackage[]{algorithm,algorithmic} 
\usepackage{mathtools}
\usepackage{tcolorbox}
\usepackage{bbm}
\usepackage{todonotes}
\usepackage{tensor}
\usepackage{caption}

\newcommand{\RR}{\mathbb{R}}

\newcommand{\epi}{\mathrm{epi\ }}
\newcommand{\hypo}{\mathrm{hypo\ }}
\newcommand{\graph}{\mathrm{graph\ }}

\newcommand{\dom}{\mathrm{dom\ }}
\newcommand{\proj}{\mathrm{proj}}
\newcommand{\interior}{\mathrm{int\ }}
\newcommand{\closure}{\mathrm{cl\ }}

\DeclareMathOperator*{\argmin}{\arg\!\min}

\mathchardef\mhyphen="2D 
\DeclareMathOperator*{\argmax}{\arg\!\max}

\newcommand{\varSpace}{\mathcal{E}}
\newcommand{\extPos}{\overline{\mathbb{R}}_{++}}

\newcommand{\radTransSup}[1]{#1^{\Gamma}}
\newcommand{\biradTransSup}[1]{#1^{\Gamma\Gamma}}

\newcommand{\radTransInf}[1]{#1_{\Gamma}}
\newcommand{\biradTransInf}[1]{#1_{\Gamma\Gamma}}
\newcommand{\radTransSet}[1]{\Gamma#1}
\newcommand{\biradTransSet}[1]{\Gamma\Gamma#1}
\definecolor{blue}{gray}{0.0}

\begin{document}

\title{Radial Duality\\ Part I: Foundations\thanks{This material is based upon work supported by the National Science Foundation Graduate Research Fellowship under Grant No. DGE-1650441. This work was partially done while the author was visiting the Simons Institute for the Theory of Computing. It was partially supported by the DIMACS/Simons Collaboration on Bridging Continuous and Discrete Optimization through NSF grant \#CCF-1740425.}
}


\author{Benjamin Grimmer
}


\institute{B. Grimmer \at
			Johns Hopkins University, Baltimore, MD.
              \email{grimmer@jhu.edu}
}

\date{Received: date / Accepted: date}

\maketitle

\begin{abstract}
Renegar~\cite{Renegar2016} introduced a novel approach to transforming generic conic optimization problems into unconstrained, uniformly Lipschitz continuous minimization.
We introduce {\it radial transformations} generalizing these ideas, equipped with an entirely new motivation and development that avoids any reliance on convex cones or functions. Of practical importance, this facilitates the development of new families of projection-free first-order methods applicable even in the presence of nonconvex objectives and constraint sets.
Our generalized construction of this radial transformation uncovers that it is dual (i.e., self-inverse) for a wide range of functions including all concave objectives. This gives a new duality relating optimization problems to their radially dual problem. For a broad class of functions, we characterize continuity, differentiability, and convexity under the radial transformation as well as develop a calculus for it.
This radial duality provides a foundation for designing projection-free radial optimization algorithms, which is carried out in the second part of this work.
\end{abstract}

\section{Introduction}
Renegar~\cite{Renegar2016} introduced a framework for conic programming (and by reduction, convex optimization), which turns such problems into uniformly Lipschitz optimization. After being radially transformed, a simple subgradient method can be applied and analyzed. Notably, even for constrained problems, such an algorithm maintains a feasible solution at each iteration while avoiding the use of orthogonal projections, which can often be a bottleneck for first-order methods.
Subsequently, Grimmer~\cite{Grimmer2017-radial-subgradient} showed that in the case of convex optimization, a simplified radial subgradient method can be applied with simpler and stronger convergence guarantees.
In~\cite{Renegar2019}, Renegar further showed that the transformation of hyperbolic cones is amenable to the application of smoothing techniques, notably improving on radial subgradient methods.

In this paper, we provide a wholly different development and generalization of the ideas behind
Renegar’s framework, which avoids relying on convex cones or functions as the central object. Instead our approach is based on the following simple projective transformation, which we dub the {\it radial point transformation}, 
\begin{equation*} 
\Gamma(x,u) = (x,1)/u
\end{equation*}
for any $(x,u)\in \varSpace\times \RR_{++}$, where $\varSpace$ is some finite-dimensional Euclidean space and $\RR_{++}=\{ u \in\mathbb{R} \mid u>0 \}$ is the set of positive real numbers. 
Applying this elementwise to a set $S\subseteq \varSpace\times\RR_{++}$ gives the {\it radial set transformation}, denoted by
\begin{equation*} 
\radTransSet{S} = \{\Gamma(x,u) \mid (x,u)\in S\}.
\end{equation*}

To motivate the nomenclature of calling these transformations radial, consider the transformation of a vertical line in $\varSpace\times\RR_{++}$: for any $x\in\varSpace$,
$$ \Gamma \{(x,\lambda) \mid \lambda\in\RR_{++}\} = \{\gamma(x,1) \mid \gamma\in\RR_{++}\}.$$
We see that this transformation maps vertical lines into rays extending from the origin (and rays into vertical lines since the point transformation is an involution (i.e., dual) $\Gamma\Gamma (x,u) = (x,u)$).

To extend this set operation to apply to functions, we consider functions $f\colon \varSpace\rightarrow \RR_{++}\cup\{0,\infty\}$ mapping into the extended positive reals (see Section~\ref{subsec:notation} for the formal definition of this range and associated notions of effective domains, epigraphs, etc). Then we define the {\it upper and lower radial function transformations} of $f$ as\footnote{Since we are considering functions mapping into $\RR_{++}\cup\{0,\infty\}$, if no $v>0$ satisfies $(y,v) \in \radTransSet{(\epi f)}$, we have the supremum defining $\radTransSup{f}(y)$ equal zero.}
\begin{equation*} 
\radTransSup{f}(y) = \sup\{v>0 \mid (y,v) \in \radTransSet{(\epi f)}\},
\end{equation*}
\begin{equation*} 
\radTransInf{f}(y) = \inf\{v>0 \mid (y,v) \in \radTransSet{(\hypo f)}\}
\end{equation*}
where $\epi f = \{(x,u)\in\varSpace\times\RR_{++} \mid f(x)\leq u\}$ denotes the epigraph of $f$ and $\hypo f = \{(x,u)\in\varSpace\times\RR_{++} \mid f(x)\geq u\}$ denotes the hypograph of $f$. The upper transformation has the interpretation of reshaping the epigraph of $f$ via the radial set transformation and returning the smallest function whose hypograph contains that set. Likewise, the lower transformation aims to turn the hypograph of $f$ into the epigraph of a new function. We find that for a wide range of functions, the duality of the point and set transformations carries over to the function transformations (see Theorem~\ref{thm:radial-function-duality}):
$$\biradTransSup{f}=f .$$

\paragraph{Connections To Prior Works.}
Noting $f^\Gamma(y) = \sup\{v>0 \mid v\cdot f(y/v) \leq 1\}$, this relates to the transformation used by Grimmer~\cite{Grimmer2017-radial-subgradient} and those of Renegar~\cite{Renegar2016,Renegar2019}. The upper and lower transformations coincide in the convex settings of these previous works but may diverge in the general setting considered herein. For our analysis, we primarily focus on the upper transformation, but equivalent results always hold for the lower transform. Artstein-Avidan and Milman~\cite{Artstein2011} (and the subsequent~\cite{Artstein2012}) consider the same underlying projective point transformation $\Gamma$ and the similar but quite different function transformation $\inf\{v>0 \mid (y,v) \in \radTransSet{(\epi f)}\}$. Considering this transformation limits their theory to the restrictive setting of nonnegative convex functions that minimize to value zero at the origin. As a result, their theory does not provide an interesting duality between optimization problems. The works~\cite{Artstein2017,Jian2013} develop calculus results similar to ours (in Section~\ref{sec:optimization}) for convex functions and use these radial operations to transform and solve various equations.

We remark on one other way to view the radial function transformation. Denote the {\it Minkowski gauge} of a set $S\subseteq\varSpace\times\RR_{++}$ at some $(y,v)\in\varSpace\times\RR_{++}$ by $\gamma_S((y,v)) = \inf\{\lambda >0 \mid (y,v)\in \lambda S\}$\footnote{Note the typical definition of this gauge allows $\lambda$ nonnegative, rather than positive. These two definitions are often equivalent, for example, for any convex $S$ containing $0$.}. Then the lower radial transformation can be restated as the restriction of this gauge to $v=1$
$$ \radTransInf{f}(y) = \gamma_{\hypo f}((y,1)).$$
This relationship to gauges motivates our notation of $\Gamma$ to denote the radial transformation. 
From this point of view, a connection can be made between this radial framework and the perspective duality considered by Aravkin et al~\cite{Aravkin2017}. They extend the theory of gauge duality developed by Freund~\cite{Freund1987}, which then applies to nonnegative convex functions by considering perspective functions. The resulting perspective dual optimization problem minimizes the function $\gamma_{\hypo f^*}((y,v))$ where $f^*$ is the Fenchel conjugate of $f$. Thus their perspective duality can be viewed as a combination of applying Fenchel duality and our radial machinery.

From this connection, we point out a key difference between this radial duality and these previous dualities. The classic theories of Lagrange and gauge duality are based on a conjugate or polar defined as a supremum over the dual vector space. In contrast, the radial dual and the Minkowski gauge are defined by a one-dimensional problem. This difference allows the radial dual to be efficiently computed numerically for generic problems using a linesearch or bisection, whereas evaluating the Fenchel conjugate of a function is as hard as optimizing over it.

\paragraph{Our Contributions.} This work serves to establish the foundations of radial transformations as a new addition to the optimization toolbox. The second part of this work leverages this machinery to develop new radial optimization algorithms. Our development establishes this tool in the following three ways: (i) The radial transformation is often dual (i.e., self-inverse) and enjoys rich structure stemming from this. (ii) The radial transformation produces a new duality between nonnegative optimization problems. For example, constraints are dually transformed into gauges, which allow algorithms to replace orthogonal projections with potentially cheaper, one-dimensional linesearches. We refer any numerically or algorithmically inclined reader to the motivating example of quadratic programming at the start of the second part of this work~\cite{Grimmer2021-part2} to see this machinery fully in action. (iii) The radial transformation is the unique operation of its kind.

\paragraph{Duality of the radial transformation.}
We precisely characterize the family of functions where the duality $\biradTransSup{f}=f$ holds through the star-convexity of their hypograph. Moreover, when this duality holds, we find that a number of important classes of functions are dual to each other or self-dual under the radial transformations. Namely,
\begin{align*}
f\text{ is upper semicontinuous } & \iff \radTransSup{f}\text{ is lower semicontinuous,}\\
f\text{ is continuous } & \iff \radTransSup{f}\text{ is continuous,}\\
f\text{ is concave } & \iff \radTransSup{f}\text{ is convex,}\\
f\text{ is quasiconcave } & \iff \radTransSup{f}\text{ is quasiconvex,}\\
f\text{ is } k \text{ times differentiable} & \iff \radTransSup{f}\text{ is } k \text{ times differentiable,}\\
f\text{ is analytic} & \iff \radTransSup{f}\text{ is analytic} 
\end{align*} 
under appropriate conditions (see Propositions~\ref{prop:semicontinuous-condition}, \ref{prop:continuous-condition}, \ref{prop:concave-conversion}, \ref{prop:quasiconcave-conversion}, and then \ref{prop:differentiable-preserved} for both differentiability conditions, respectively). We also derive a calculus for the radial transformations, providing formulas for the (sub)gradients and Hessians of $\radTransSup{f}$ based on those of $f$.

\paragraph{Radial duality between optimization problems.}
For a wide range of functions, $\Gamma(\epi f)$ is the hypograph of another function, and so
$ \hypo f^\Gamma = \Gamma(\epi f).$
As a result, for such functions, points in $\hypo f$ and $\epi f^\Gamma$ can be directly related by the bijection $\Gamma$ and its inverse (which is also $\Gamma$). This relation also applies to the maximizers of $f$ and the minimizers of $f^\Gamma$. Namely
for any function $f\colon \varSpace\rightarrow \RR_{++}\cup\{0,\infty\}$, consider the primal problem
\begin{equation}
p^* = \max_{x\in\varSpace} f(x). \label{eq:base-problem}
\end{equation}
Then the radially dual problem, given by
\begin{equation}
d^* = \min_{y\in\varSpace} \radTransSup{f}(y), \label{eq:radial-problem}
\end{equation}
has $\argmax f\times\{p^*\}= \Gamma\left(\argmin \radTransSup{f}\times\{d^*\}\right)$ under regularity conditions (see Proposition~\ref{prop:minimizer-characterization}). Hence maximizing $f$ is equivalent to minimizing $\radTransSup{f}$. 

The radially dual problem~\eqref{eq:radial-problem} can exhibit very different behavior than the original problem~\eqref{eq:base-problem}. 
For example, consider the function $f(x) = \sqrt{1 - \|x\|^2}$ defined on the unit ball, which has arbitrarily large gradients and Hessians as $x$ approaches the boundary of the ball. Despite this function's poor behavior, $\radTransSup{f}(y) = \sqrt{1+\|y\|^2}$ has full domain with gradients and Hessians bounded in norm by one everywhere. This structure is very appealing for the analysis of first-order optimization methods which tend to heavily rely on these quantities being bounded.
The second part of this work utilizes such structure arising from the radial duality developed herein to propose and analyze projection-free radial optimization methods. 

\paragraph{Uniqueness of the radial transformation.}
From our construction of the radial transformation, it is natural to ask if other interesting transformations of optimization problems can be given by reshaping the epigraph of a function. Under some basic assumptions (primarily that the reshaping is invertible and convexity preserving), there are only two transformations of this form, up to affine transformations: the trivial duality between maximizing a function and minimizing its negative and the nontrivial duality given by the radial transformation. These results are similar in spirit to the characterization of order isomorphisms by~\cite{Artstein2012}.

\paragraph{Outline}
In the remainder of this introduction, we sketch the usage of this radial machinery for constrained optimization and introduce needed notations. Section~\ref{sec:set-transform} develops theory for the radial point and set transformations on $\varSpace\times \RR_{++}$. Informed by this, Section~\ref{sec:function-transform} derives the core theory establishing the radial function transformations. Then Section~\ref{sec:optimization} develops the calculus and optimality relationships between the primal~\eqref{eq:base-problem} and radial dual~\eqref{eq:radial-problem}. Lastly, Section~\ref{sec:axioms} shows that this radial framework is the unique transformation of nonnegative-valued optimization problems of its type.

The second part of this work~\cite{Grimmer2021-part2} uses the radial theory developed here to design new ``radial'' first-order optimization methods. These methods avoid assumptions of Lipschitz continuity and the use orthogonal projections all together instead only relying on cheaper gauge evaluations. Convergence theory and some promising preliminary numerical results are presented there.

\subsection{Application to Constrained Optimization}
One useful facet of the duality between~\eqref{eq:base-problem} and~\eqref{eq:radial-problem} lies in how constraints are transformed and the subsequent algorithmic gains. We sketch these consequences here, which are further explored in the second part of this work.
Consider maximizing $f\colon\varSpace \rightarrow\extPos$ over some $S\subseteq \varSpace$ where  $\extPos = \RR_{++}\cup\{0,+\infty\}$.
Defining the following nonstandard indicator function of a set $S\subseteq \varSpace$ as
$$\hat \iota_S (x) = \begin{cases}
+\infty & \text{if\ } x\in S\\
0 & \text{if\ } x\not\in S,
\end{cases}$$
consider the constrained primal optimization problem
$$ \max_{x\in S} f(x) = \max_{x\in\varSpace} \min\{f(x),\hat \iota_S(x)\}.$$
Suppose $S$ is convex with $0\in S$. In this case, this nonstandard indicator function has $\hat \iota_S^\Gamma(y)$ equal its gauge $\gamma_S(y) = \inf\{\lambda> 0 \mid y\in\lambda S\}$.
Calculating the radially dual problem~\eqref{eq:radial-problem} (via Proposition~\ref{prop:radial-formulas}) yields
$$ \min_{y\in\varSpace} \max\{\radTransSup{f}(y), \gamma_S(y)\}.$$

This reformulation motivates the design of new first-order methods. Many methods for constrained optimization problems require orthogonal projections onto the feasible region at each iteration (which may be substantially more expensive than computing a single (sub)gradient, often dominating an algorithm's runtime). Evaluating the gauge of a set and computing one of its subgradients can be far cheaper than orthogonal projection as it requires at most a one-dimensional line search and computing a single normal vector of the constraint set. Observing and taking advantage of this structure in the context of conic programming was a central contribution of~\cite{Renegar2016}. 

For example, consider a generic polyhedron $S=\{x\in\varSpace \mid a_i^Tx\leq b_i\}$ with $b>0$ (which is equivalent to having $0$ lie in the interior of $S$). The gauge of such a polyhedron equals $\gamma_S(x) = \max\{a_i^Tx/b_i\}$, for which the function value and subgradients can be computed by a single matrix multiplication whereas orthogonal projection requires solving a quadratic program. As a second example, consider the translated cone of positive semidefinite matrices $S = \{X \mid X+I \succeq 0 \}$. The gauge of such a set is $\gamma_S(X) = \max\{-\lambda_{min}(X),0\}$, whose evaluation requires a minimum eigenvalue computation and subgradients follow from computing a minimum eigenvector. In contrast, orthogonal projection onto this cone requires a full spectral decomposition.

\subsection{Notation} \label{subsec:notation}
We primarily consider sets in $\varSpace \times \RR_{++}$, which inherits the standard Euclidean inner product and norm from $\varSpace\times \RR$. Denote the ball of radius $r>0$ around a point $(x,u)\in\varSpace\times\RR$ by
$$B((x,u),r) :=\{(x',u')\in\varSpace\times\RR \mid \|(x',u') - (x,u)\|\leq r\}.$$
Further, denote orthogonal projection onto a closed set $S \subseteq \varSpace\times \RR$ by
$$ \proj_S((x,u)) := \argmin\{\|(x',u') - (x,u)\| \mid (x',u')\in S\}.$$
Note $\proj_S$ is set valued and may not be a singleton if $S$ is not convex.

We consider functions $f \colon \varSpace \rightarrow \extPos$, where $\extPos = \RR_{++}\cup\{0,+\infty\}$ denotes the ``extended positive reals''. Here $0$ and $+\infty$ are the limit objects of $\RR_{++}$, mirroring the roles of $-\infty$ and $+\infty$ in the extended reals.
The effective domain of such a function is denoted by
$$\dom f := \{x\in\varSpace \mid f(x)\in\RR_{++}\}.$$
Such functions relate to $\varSpace \times \RR_{++}$ through their graphs, epigraphs, and hypographs
\begin{align*}
\graph f &:= \{(x,u)\in \varSpace \times \RR_{++} \mid f(x)= u\},\\
\epi f &:= \{(x,u)\in \varSpace \times \RR_{++} \mid f(x)\leq u\},\\
\hypo f &:= \{(x,u)\in \varSpace \times \RR_{++} \mid f(x)\geq u\}.
\end{align*}

We say a function $f \colon \varSpace \rightarrow \extPos$ is upper (lower) semicontinuous if $\hypo f$ ($\epi f$) is closed with respect to $\varSpace \times \RR_{++}$.
Equivalently, a function is upper semicontinuous if for all $x\in\varSpace$, $f(x)=\limsup_{x'\rightarrow x} f(x')$ and lower semicontinuous if $f(x)=\liminf_{x'\rightarrow x} f(x')$.

We say a function $f \colon \varSpace \rightarrow \extPos$ is concave (convex) if $\hypo f$ ($\epi f$) is convex.
The set of {\it convex normal vectors} of $S\subseteq \varSpace\times \RR$ at some $(x,u)\in S$ is
$$ N^C_{S}((x,u)) := \{(\zeta,\delta) \mid (\zeta,\delta)^T((x,u) - (x',u')) \geq 0\ \forall (x',u')\in S\}.$$
Then the {\it convex subdifferential} of a function $f$ at some $x\in\dom f$ is
$$ \partial_C f(x) := \{\zeta \mid (\zeta,-1)\in N^C_{\epi f}((x,f(x)))\}.$$
Likewise, the {\it convex supdifferential} of a function $f$ at some $x\in\dom f$ is
$$ \partial^C f(x) := \{\zeta \mid (-\zeta,1)\in N^C_{\hypo f}((x,f(x)))\}.$$
The elements of these differentials are referred to as convex subgradients or supgradients.

For sets and functions that are not convex, we consider the generalization given by proximal normals and differentials. The set of {\it proximal normal vectors} of a set $S\subseteq \varSpace\times \RR$ at some $(x,u)\in S$ is 
$$ N^P_{S}((x,u)) := \{(\zeta,\delta) \mid (x,u)\in \proj_S( (x,u) + \epsilon(\zeta,\delta)) \text{\ for some\ }\epsilon>0\}.$$
Then the {\it proximal subdifferential} of a function $f$ at some $x\in\dom f$ is
$$ \partial_P f(x) := \{\zeta \mid (\zeta,-1)\in N^P_{\epi f}((x,f(x)))\}.$$
Likewise, the {\it proximal supdifferential} of a function $f$ at some $x\in\dom f$ is
$$ \partial^P f(x) := \{\zeta \mid (-\zeta,1)\in N^P_{\hypo f}((x,f(x)))\}.$$
The elements of these differentials are referred to as proximal subgradients or supgradients.

\section{The Radial Set Transformation} \label{sec:set-transform}
We begin by observing a number of properties of the radial point and set transformations. 
Section~\ref{subsec:radial-normals} uses these to characterize the convex and proximal normal vectors of a radially transformed set.
Then Section~\ref{subsec:radial-set-examples} concludes with a number of examples and pictures illustrating the radial set transformation.
A careful understanding of this operation on sets forms the foundation for understanding the radial function transformation. 

One can easily check the point transformation is a continuous analytic bijection on $\varSpace\times\RR_{++}$. Further, both the point and set transformations are dual since
\begin{equation}
\biradTransSet{(x,u)} =\radTransSet{\frac{(x,1)}{u}}= \frac{(x/u, 1)}{1/u} = (x,u).\label{eq:radial-set-duality}
\end{equation}

Now we observe a few basic properties of the set transformation on any pair of sets $S,T\subseteq\varSpace\times\RR_{++}$.
First, since the point transformation is invertible (in fact, it is its own inverse), the set transformation preserves inclusions between sets, giving
\begin{equation}
S\subseteq T \iff \radTransSet{S} \subseteq \radTransSet{T}. \label{eq:radial-inclusion}
\end{equation}
The radial set transformation distributes over unions and intersections, giving
\begin{equation}
\radTransSet{(S \cap T)} = \radTransSet{S} \cap \radTransSet{T}, \label{eq:radial-intersection}
\end{equation}
\begin{equation}
\radTransSet{(S \cup T)} = \radTransSet{S} \cup \radTransSet{T}. \label{eq:radial-union}
\end{equation}

Since the radial point transformation is a projective transformation, convex sets, halfspaces, and ellipsoids map into convex sets, halfspaces, and ellipsoids, respectively. We give direct proofs of these results (see Proposition~\ref{prop:convex-set-radial}, \ref{prop:halfspace-radial}, and~\ref{prop:ellipsoid-radial}) in the appendix yielding formulas for radially dual halfspaces and ellipsoids in the latter two cases. 
\begin{proposition}\label{prop:convex-set-radial}
	A set $S\subseteq\varSpace\times\RR_{++}$ is convex if and only if $\radTransSet{S}$ is convex.
\end{proposition}

In particular, consider the radial transformation of any halfspace in $\varSpace\times \RR_{++}$. Direct manipulation of its definition shows that the radial transformation of a halfspace is another halfspace.
\begin{proposition}\label{prop:halfspace-radial}
	A set $S\subseteq \varSpace\times\RR_{++}$ is a halfspace if and only if $\radTransSet{S}$ is a halfspace.
	In particular, for any
	$S = \left\{(x',u')\in\varSpace\times\RR_{++} \mid \begin{bmatrix} \zeta \\ \delta 
	\end{bmatrix}^T\begin{bmatrix} x'-x \\ u'-u 
	\end{bmatrix}\leq 0\right\},$
	letting $(y,v) = \Gamma(x,u)$, $\radTransSet{S}$ is the following halfspace
	$$\radTransSet{S} = \left\{(y',v')\in\varSpace\times\RR_{++} \mid  \begin{bmatrix} \zeta \\ -(\zeta,\delta)^T(x,u)
	\end{bmatrix}^T\begin{bmatrix} y'-y \\ v'-v 
	\end{bmatrix} \leq 0 \right\}.$$
\end{proposition}

We say that a set is {\it polyhedral} if it is the intersection of finitely many halfspaces and $\varSpace\times\RR_{++}$. Then as an immediate consequence of Proposition~\ref{prop:halfspace-radial} and~\eqref{eq:radial-intersection}, being polyhedral is preserved under the radial set transformation.
\begin{corollary}\label{cor:polyhedral-sets}
	A set $S\subseteq \varSpace\times\RR_{++}$ is polyhedral if and only if $\radTransSet{S}$ is polyhedral.
\end{corollary}

Lastly, we consider the radial transformation of ellipsoids. A set $S\subseteq\varSpace\times\RR$ is an ellipsoid if for some center $(x,u)$ and positive definite linear mapping $H$,
\begin{equation}
S = \left\{(x',u')\in\varSpace\times\RR \mid \begin{bmatrix} x'-x \\ u'-u \end{bmatrix}^T H \begin{bmatrix} x'-x \\ u'-u \end{bmatrix} \leq 1\right\}. \label{eq:basic-ellipsoid}
\end{equation}
Similar to halfspaces, the radial transformation of such an ellipsoid in $\varSpace\times\RR_{++}$ is an ellipsoid in $\varSpace\times\RR_{++}$. Curiously, the center of $\radTransSet{S}$ is not $\Gamma(x,u)$ (as one might expect), but rather depends on $H$.

\begin{proposition}\label{prop:ellipsoid-radial}
	A set $S\subseteq \varSpace\times\RR_{++}$ is an ellipsoid if and only if $\radTransSet{S}$ is an ellipsoid.
\end{proposition}

\subsection{Normal Vectors Under the Radial Set Transformation}\label{subsec:radial-normals}
Next we consider how the normal vectors of a set relate to those of its radial transformation. Proposition~\ref{prop:halfspace-radial}'s description of transformed halfspaces characterizes convex normal vectors under the transformation. Combining this result with Proposition~\ref{prop:ellipsoid-radial}'s description of transformed ellipsoids gives a characterization for proximal normal vectors.
\begin{proposition}\label{prop:convex-normals}
	For any $S\subseteq\varSpace\times\RR_{++}$, all $(y,v)\in \Gamma S$ have
	$$N^C_{\Gamma S}((y,v)) = \left\{\begin{bmatrix} \zeta \\ -(\zeta,\delta)^T(x,u)
	\end{bmatrix} \mid \begin{bmatrix} \zeta \\ \delta
	\end{bmatrix}\in N^C_{S}((x,u))\right\}$$
	where $(x,u) = \Gamma(y,v)$.
\end{proposition}
\begin{proof}
	For any $(x,u)\in S$, $(\zeta,\delta)\in N^C_{S}((x,u))$ if and only if
	$$S \subseteq \left\{(x',u')\in\varSpace\times\RR_{++} \mid \begin{bmatrix} \zeta \\ \delta 
	\end{bmatrix}^T\begin{bmatrix} x'-x \\ u'-u 
	\end{bmatrix}\leq 0\right\}.$$
	Letting $(y,v) = \Gamma(x,u)$, Proposition~\ref{prop:halfspace-radial} and~\eqref{eq:radial-inclusion} imply
	$$\Gamma S \subseteq \left\{(y',v')\in\varSpace\times\RR_{++} \mid  \begin{bmatrix} \zeta \\ -(\zeta,\delta)^T(x,u)
	\end{bmatrix}^T\begin{bmatrix} y'-y \\ v'-v 
	\end{bmatrix} \leq 0 \right\}.$$
	Thus $\left(\zeta,\ -(\zeta,\delta)^T(x,u)\right)\in N^C_{\Gamma S}((y,v))$. This gives the containment
	$$N^C_{\Gamma S}((y,v)) \supseteq \left\{\begin{bmatrix} \zeta \\ -(\zeta,\delta)^T(x,u)
	\end{bmatrix} \mid \begin{bmatrix} \zeta \\ \delta
	\end{bmatrix}\in N^C_{S}((x,u))\right\}$$
	and repeating the argument, replacing $S$ by $\Gamma S$, gives
	$$N^C_{S}((x,u))=N^C_{\Gamma\Gamma S}((x,u)) \supseteq \left\{\begin{bmatrix} \zeta' \\ -(\zeta',\delta')^T(y,v)
	\end{bmatrix} \mid \begin{bmatrix} \zeta' \\ \delta'
	\end{bmatrix}\in N^C_{\Gamma S}((y,v))\right\}.$$
	Applying these two containments in succession shows the claimed formula
	\begin{align*}
	N^C_{\Gamma S}((y,v)) &\supseteq \left\{\begin{bmatrix} \zeta \\ -(\zeta,\delta)^T(x,u)
	\end{bmatrix} \mid \begin{bmatrix} \zeta \\ \delta
	\end{bmatrix}\in N^C_{S}((x,u))\right\}\\
	&\supseteq \left\{\begin{bmatrix} \zeta' \\ -(\zeta',-(\zeta',\delta')^T(y,v))^T(x,u)
	\end{bmatrix} \mid \begin{bmatrix} \zeta' \\ \delta'
	\end{bmatrix}\in N^C_{\Gamma S}((y,v))\right\}\\
	&=N^C_{\Gamma S}((y,v)). \tag*{\qed}
	\end{align*}
\end{proof}

\begin{proposition}\label{prop:proximal-normals}
	For any $S\subseteq\varSpace\times\RR_{++}$, all $(y,v)\in \Gamma S$ have
	$$N^P_{\Gamma S}((y,v)) = \left\{\begin{bmatrix} \zeta \\ -(\zeta,\delta)^T(x,u)
	\end{bmatrix} \mid \begin{bmatrix} \zeta \\ \delta
	\end{bmatrix}\in N^P_{S}((x,u))\right\}$$
	where $(x,u) = \Gamma(y,v)$.
\end{proposition}
\begin{proof}
	Consider any $(x,u)\in S$ and $(\zeta,\delta)\in N^P_{S}((x,u))$. Then for some $\epsilon>0$, the ball
	$$E = B\left(\begin{bmatrix} x \\ u\end{bmatrix}+\epsilon\begin{bmatrix} \zeta \\ \delta\end{bmatrix},\ \epsilon\left\|\begin{bmatrix} \zeta \\ \delta\end{bmatrix}\right\|\right)\subset \varSpace\times\RR_{++}$$
	has $E\cap S = \{(x,u)\}$. 
	Recall from Proposition~\ref{prop:ellipsoid-radial} that $\Gamma E$ is an ellipsoid.
	Applying~\eqref{eq:radial-intersection} implies $\Gamma E \cap \Gamma S = \{(y,v)\}$ where $(y,v) = \Gamma(x,u)$.
	Since $-(\zeta,\delta)\in N^C_{E}((x,u))$, Proposition~\ref{prop:convex-normals} implies $\left(-\zeta,\ (\zeta,\delta)^T(x,u)\right)\in N^C_{\Gamma E}((y,v))$.
	Then for sufficiently small $\epsilon'>0$, the ball
	$$E'=B\left(\begin{bmatrix} y \\ v\end{bmatrix}+\epsilon'\begin{bmatrix} \zeta \\ -(\zeta,\delta)^T(x,u)
	\end{bmatrix},\ \epsilon'\left\|\begin{bmatrix} \zeta \\ -(\zeta,\delta)^T(x,u)
	\end{bmatrix}\right\|\right)$$
	lies in $\Gamma E$, and hence has $E'\cap \Gamma S = \{(y,v)\}$. Thus $\left(\zeta,\ -(\zeta,\delta)^T(x,u)\right)\in N^P_{\Gamma S}((y,v))$, and so
	$$N^P_{\Gamma S}((y,v)) \supseteq \left\{\begin{bmatrix} \zeta \\ -(\zeta,\delta)^T(x,u)
	\end{bmatrix} \mid \begin{bmatrix} \zeta \\ \delta
	\end{bmatrix}\in N^P_{S}((x,u))\right\}.$$
	As shown in the proof of Proposition~\ref{prop:convex-normals}, the claimed formula follows from this containment and the dual containment given by replacing $S$ by $\Gamma S$. \qed
\end{proof}
\subsection{Examples and Pictures}\label{subsec:radial-set-examples}
In Figures~\ref{fig:set-p1} through~\ref{fig:set-d5}, we give five examples of pairs of sets in $\RR\times \RR_{++}$ radially dual to each other. Each figure includes the horizontal line $L = \{(x,1)\mid x\in\RR\}$ as a black dashed line. Observe that $L$ is exactly the set of fixed points of the radial point transformation. Further, points above $L$ always map into points below $L$ (and vice versa).

The first two example pairs given in Figures~\ref{fig:set-p1} and~\ref{fig:set-d1} and Figures~\ref{fig:set-p2} and~\ref{fig:set-d2} show the radial transformation of a halfspace and a polyhedron (which must be a halfspace and a polyhedron by Proposition~\ref{prop:halfspace-radial} and Corollary~\ref{cor:polyhedral-sets}). Examining the transformation of the horizontal and vertical faces of the square in Figure~\ref{fig:set-p2} demonstrates two simple properties of the radial set transformation: (i) horizontal lines map into horizontal lines and (ii) vertical lines map into rays extending away from the origin (and vice versa).

Figures~\ref{fig:set-p3} and~\ref{fig:set-d3} show the radial transformation of an ellipsoid (which must be an ellipsoid by Proposition~\ref{prop:ellipsoid-radial}). Figures~\ref{fig:set-p4} and~\ref{fig:set-d4} consider the radial set transformation of a parabola, which is nearly an ellipsoid in $\RR\times\RR_{++}$ but it approaches height $0$ at the origin.

Our last pair of examples in Figures~\ref{fig:set-p5} and~\ref{fig:set-d5} show the radial set transformation of a sine wave. Notice that the resulting set is not the graph of any function. As we now transition to discussing our radial function transformations, considering how graphs, epigraphs, and hypographs behave under the set transformation provides key intuitions. The fact that the epigraph of our example parabola does not transform into the hypograph of another function and that the graph of our example sine wave does not transform into the graph of another function (as we will see) correspond to radial duality not holding for these function.

\begin{figure}
	\centering
	\begin{minipage}{.4\textwidth}
		\centering
		\includegraphics[width=.99\linewidth]{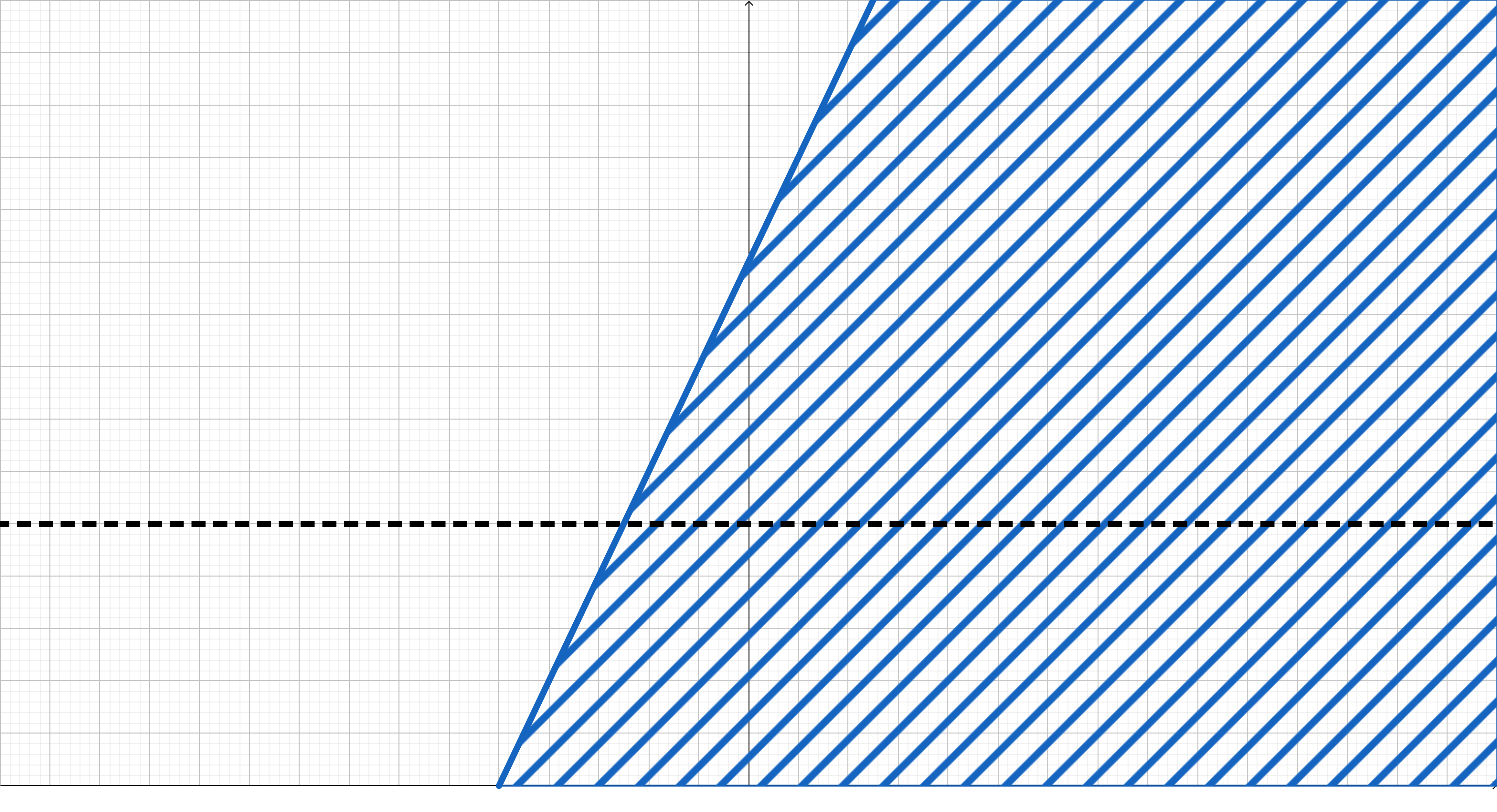}
		\captionof{figure}{A halfspace.}
		\label{fig:set-p1}
	\end{minipage}%
	$\iff$
	\begin{minipage}{.4\textwidth}
		\centering
		\includegraphics[width=.99\linewidth]{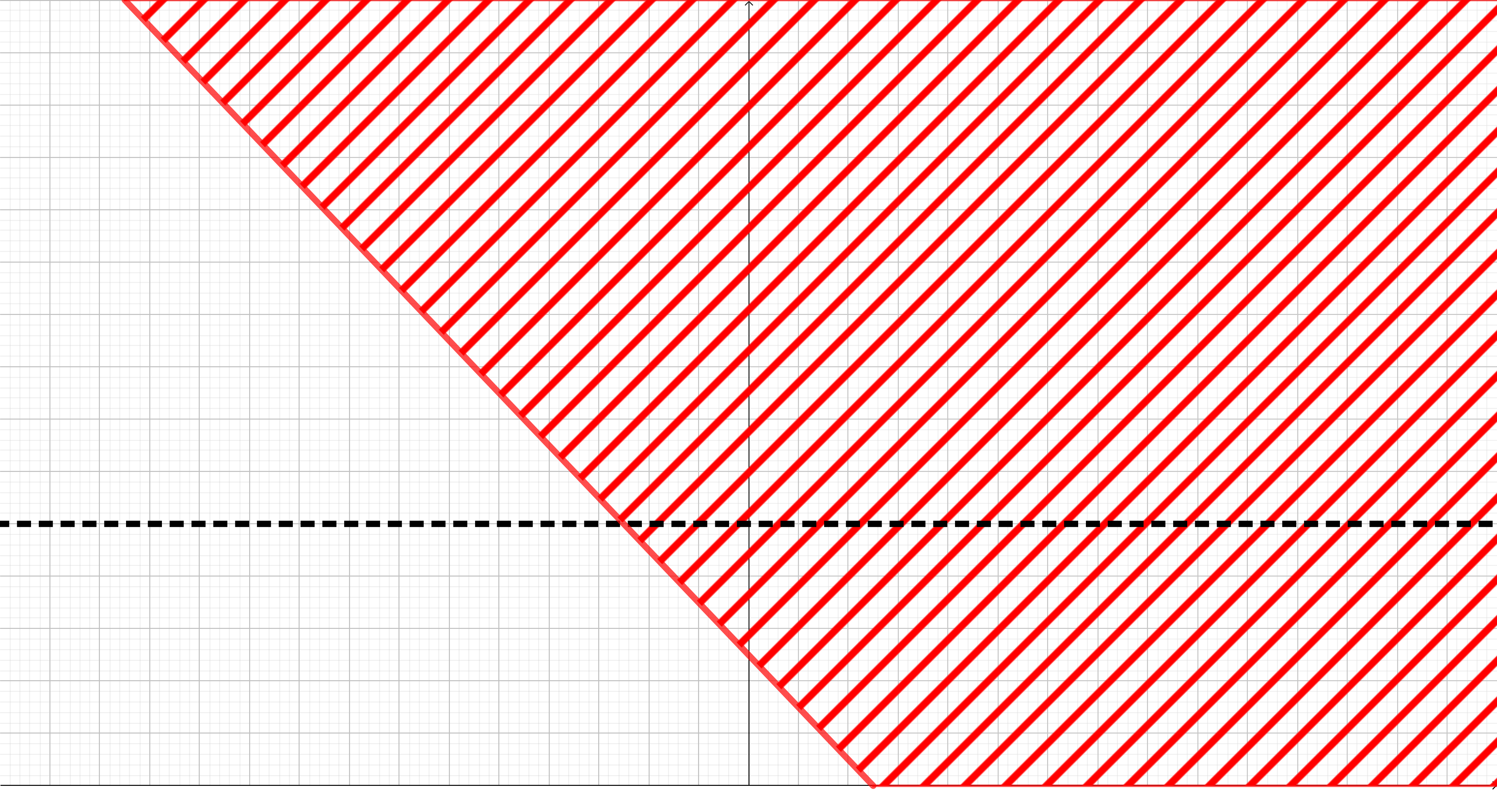}
		\captionof{figure}{Dual halfspace.}
		\label{fig:set-d1}
	\end{minipage}
	
	\begin{minipage}{.4\textwidth}
		\centering
		\includegraphics[width=.99\linewidth]{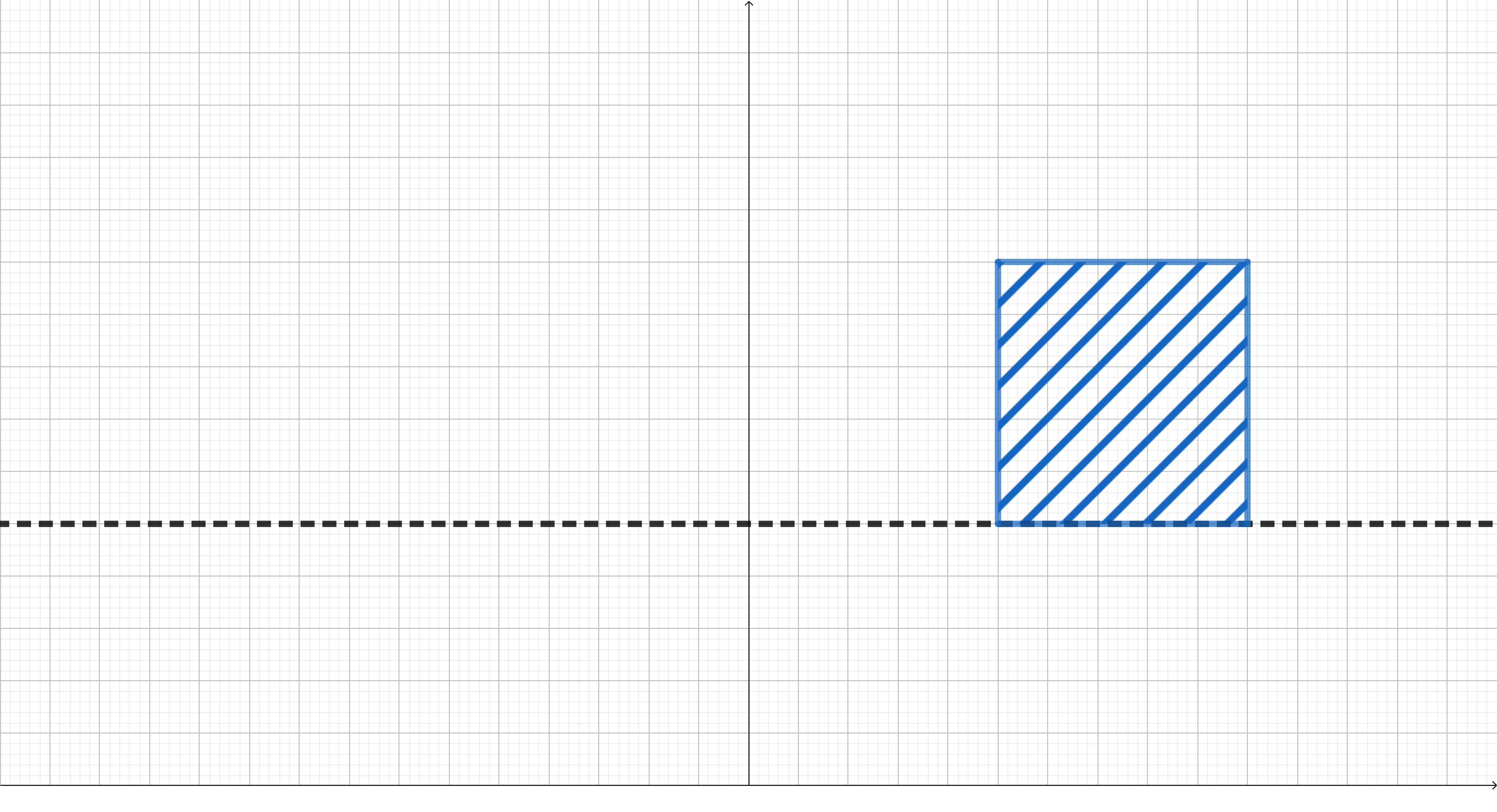}
		\captionof{figure}{A polyhedron.}
		\label{fig:set-p2}
	\end{minipage}%
	$\iff$
	\begin{minipage}{.4\textwidth}
		\centering
		\includegraphics[width=.99\linewidth]{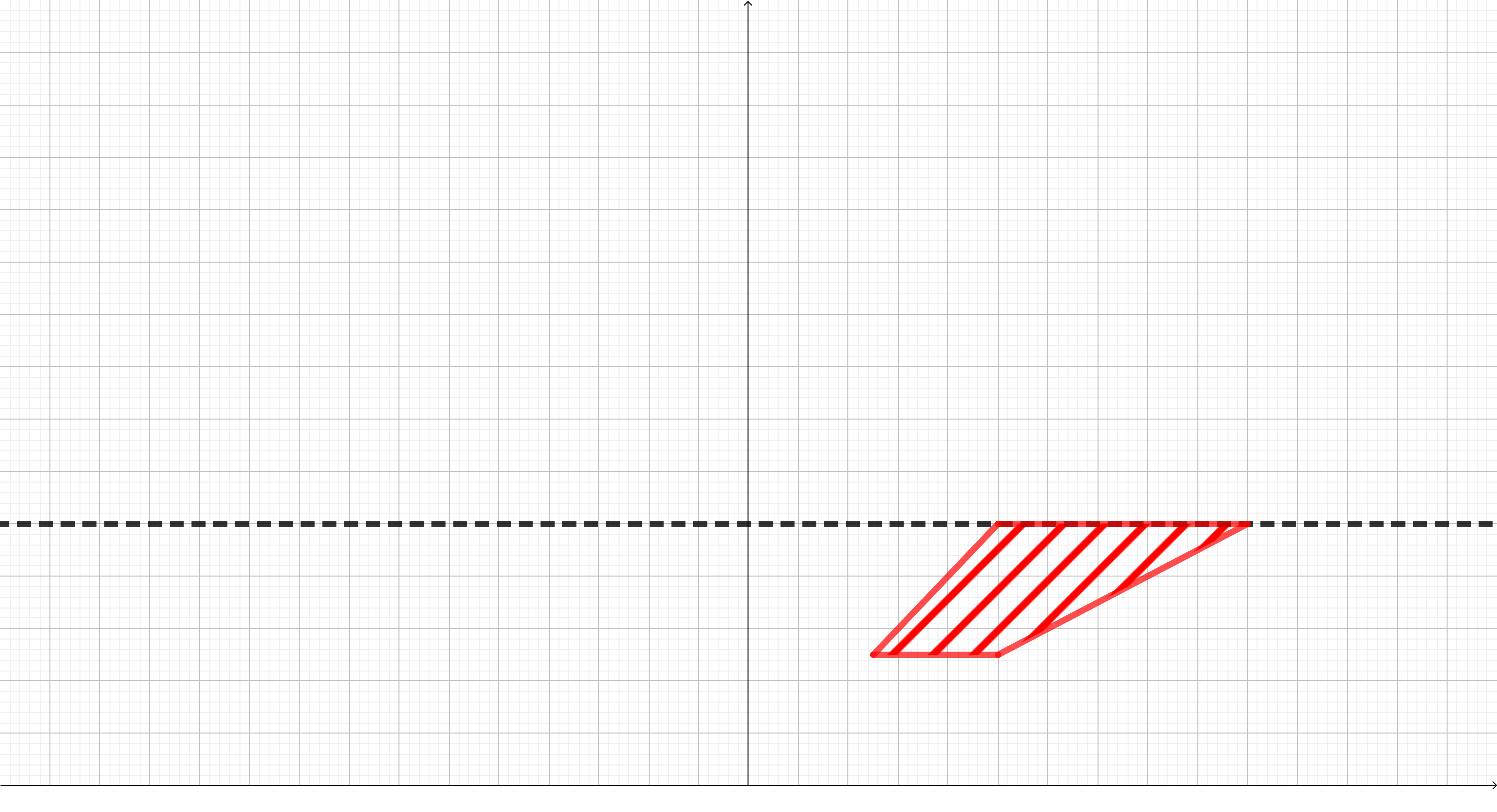}
		\captionof{figure}{Dual polyhedron.}
		\label{fig:set-d2}
	\end{minipage}
	
	\begin{minipage}{.4\textwidth}
		\centering
		\includegraphics[width=.99\linewidth]{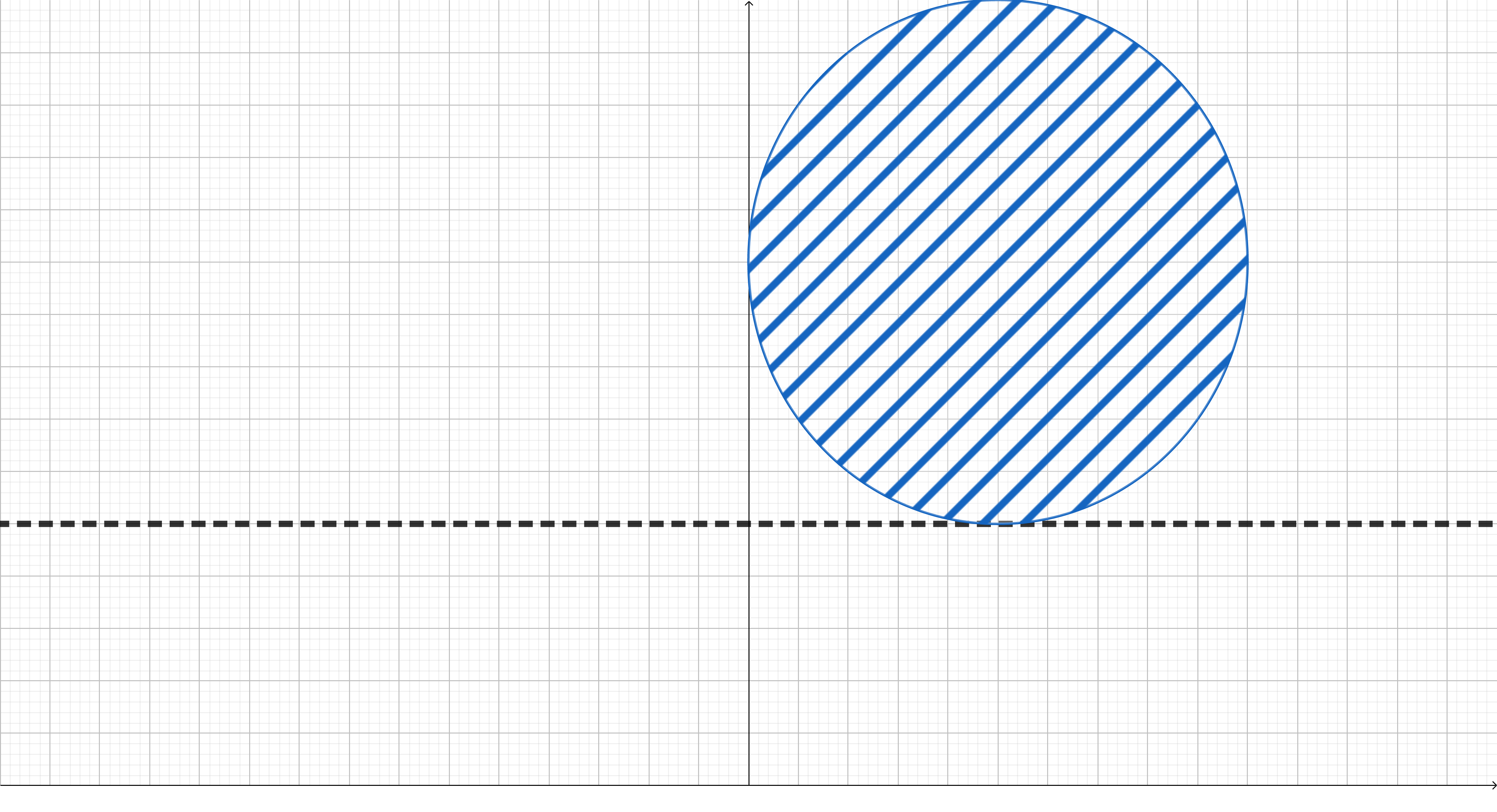}
		\captionof{figure}{An ellipsoid.}
		\label{fig:set-p3}
	\end{minipage}%
	$\iff$
	\begin{minipage}{.4\textwidth}
		\centering
		\includegraphics[width=.99\linewidth]{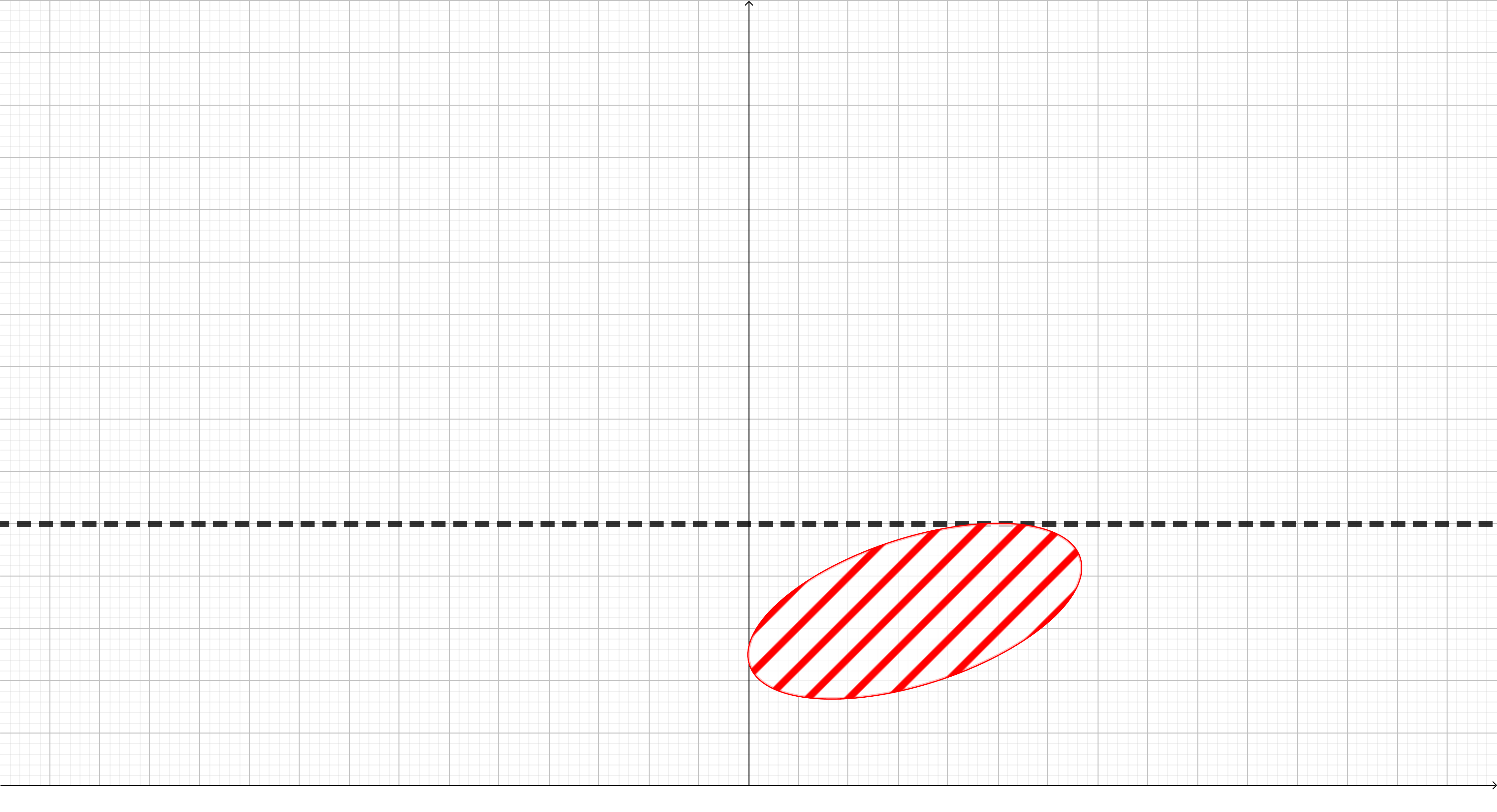}
		\captionof{figure}{Dual ellipsoid.}
		\label{fig:set-d3}
	\end{minipage}
	
	\begin{minipage}{.4\textwidth}
		\centering
		\includegraphics[width=.99\linewidth]{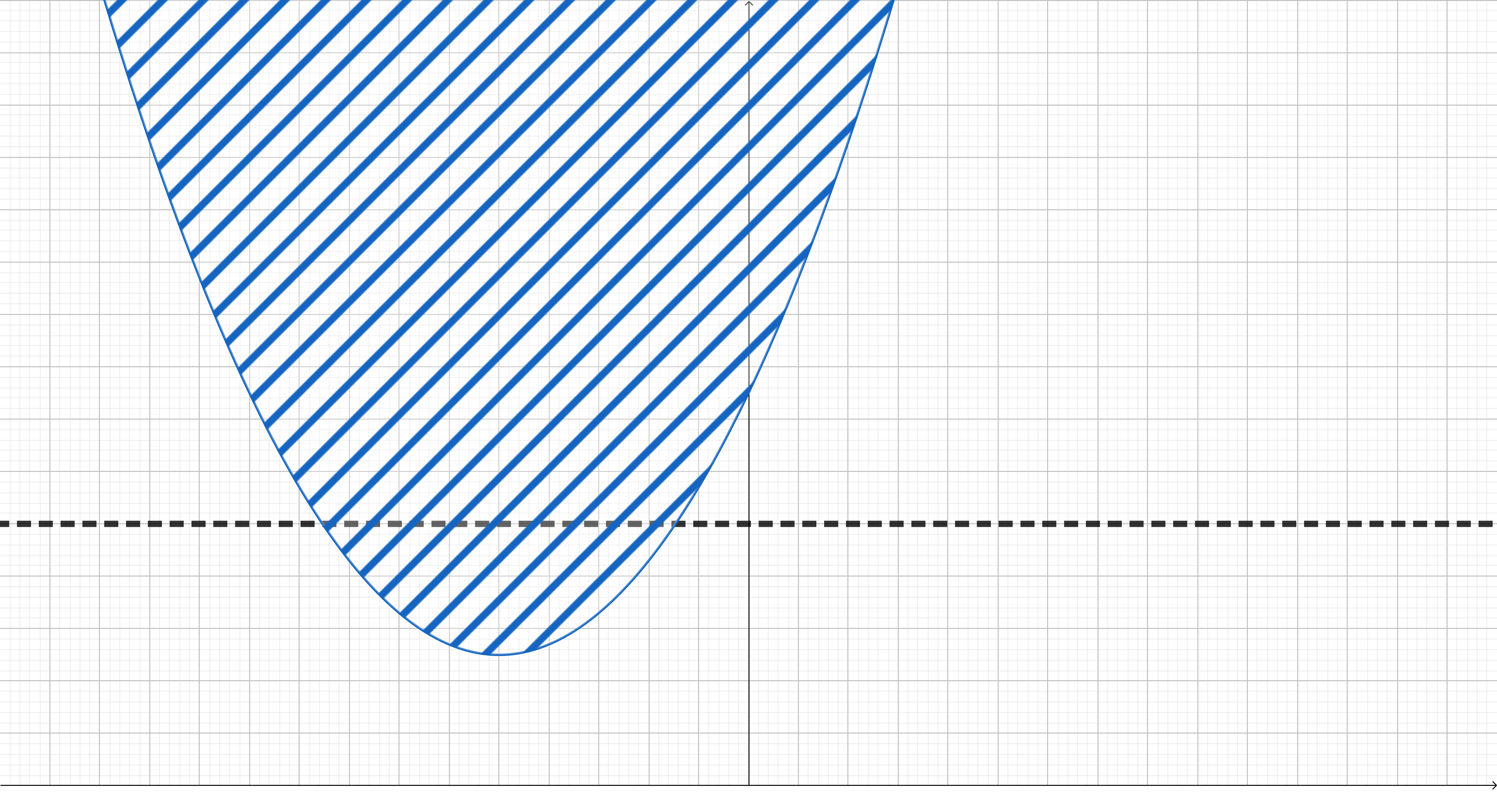}
		\captionof{figure}{A quadratic.}
		\label{fig:set-p4}
	\end{minipage}%
	$\iff$
	\begin{minipage}{.4\textwidth}
		\centering
		\includegraphics[width=.99\linewidth]{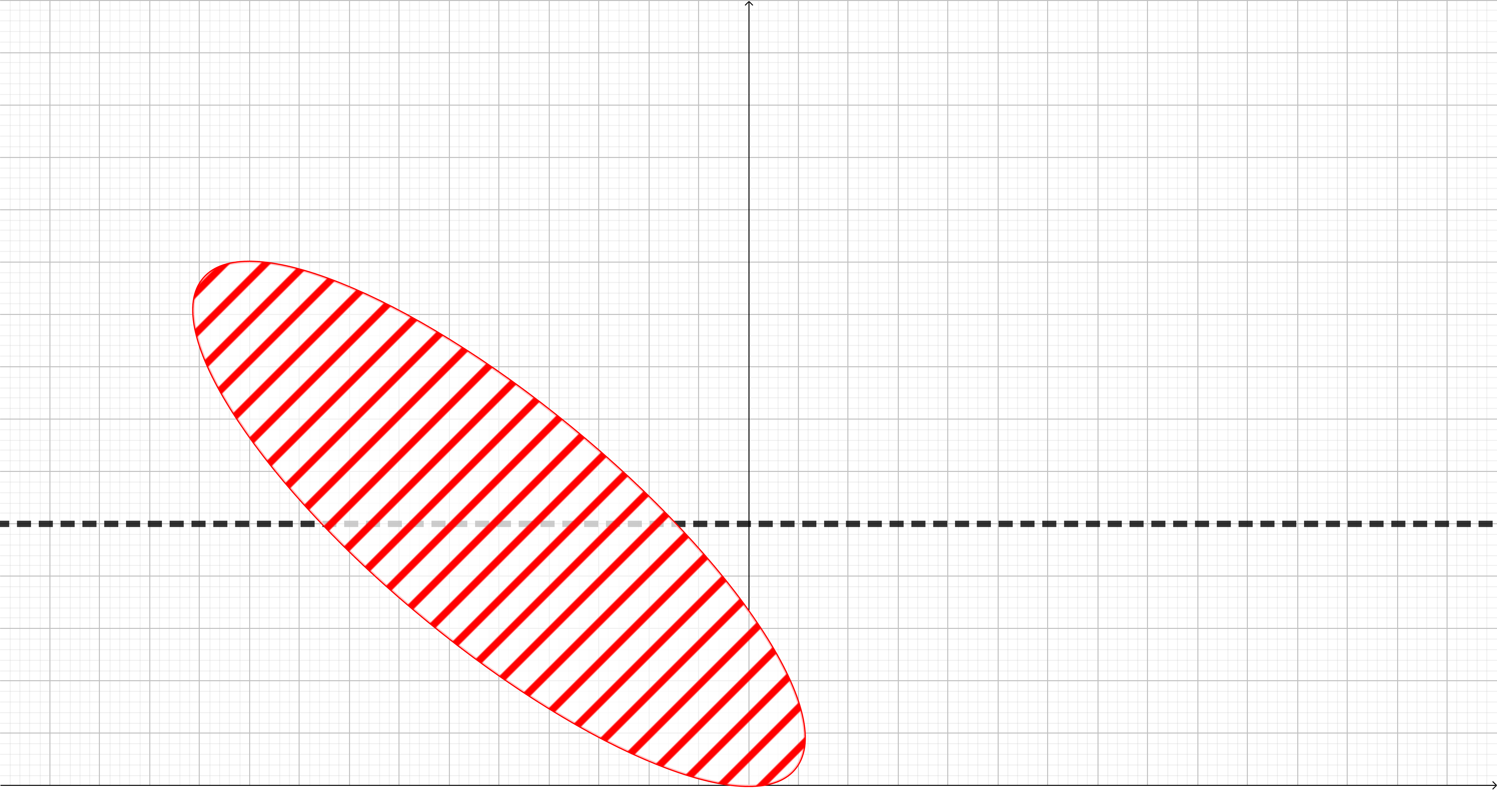}
		\captionof{figure}{Dual of a quadratic.}
		\label{fig:set-d4}
	\end{minipage}
	
	\begin{minipage}{.4\textwidth}
		\centering
		\includegraphics[width=.99\linewidth]{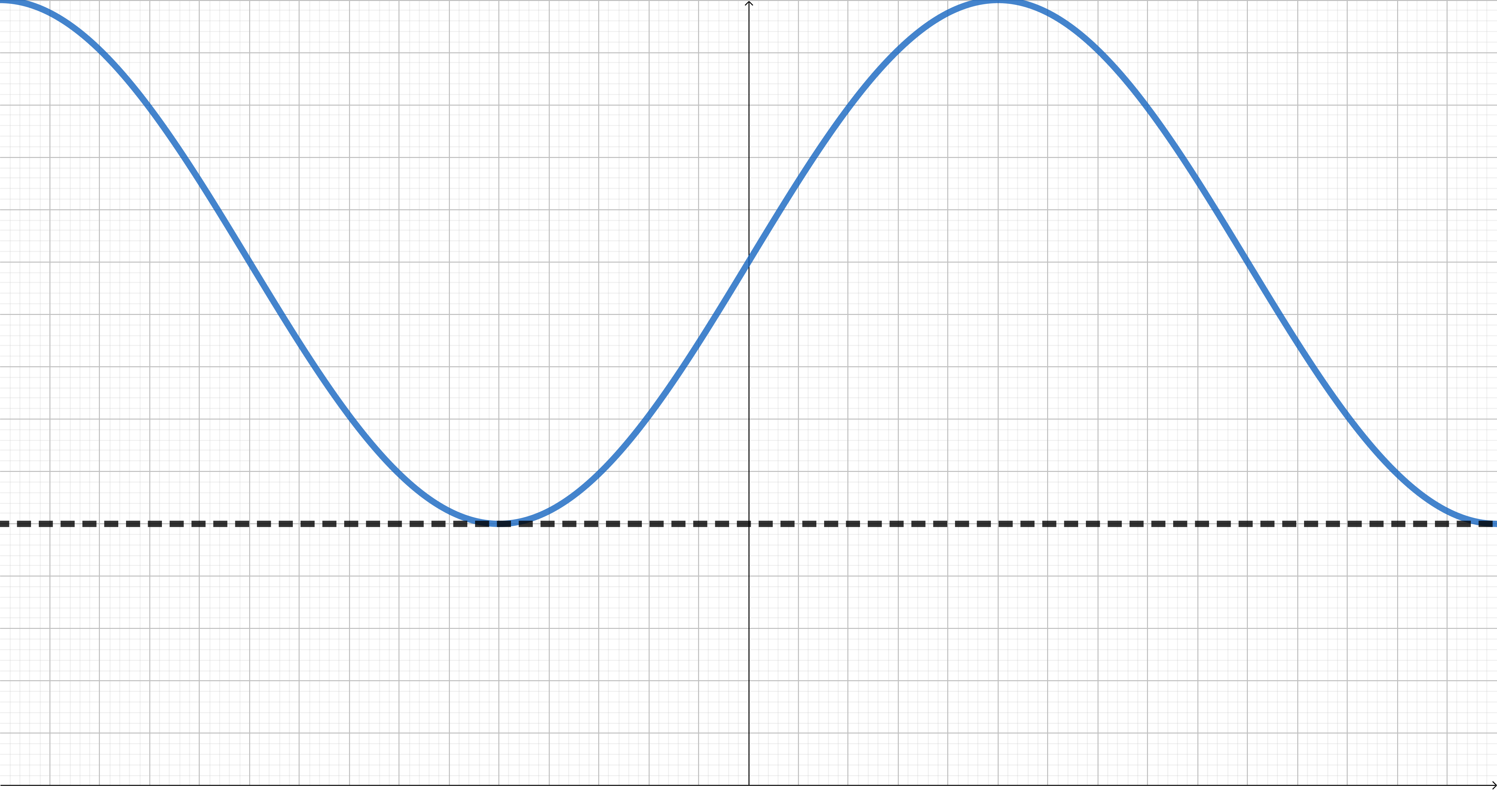}
		\captionof{figure}{A sine wave.}
		\label{fig:set-p5}
	\end{minipage}%
	$\iff$
	\begin{minipage}{.4\textwidth}
		\centering
		\includegraphics[width=.99\linewidth]{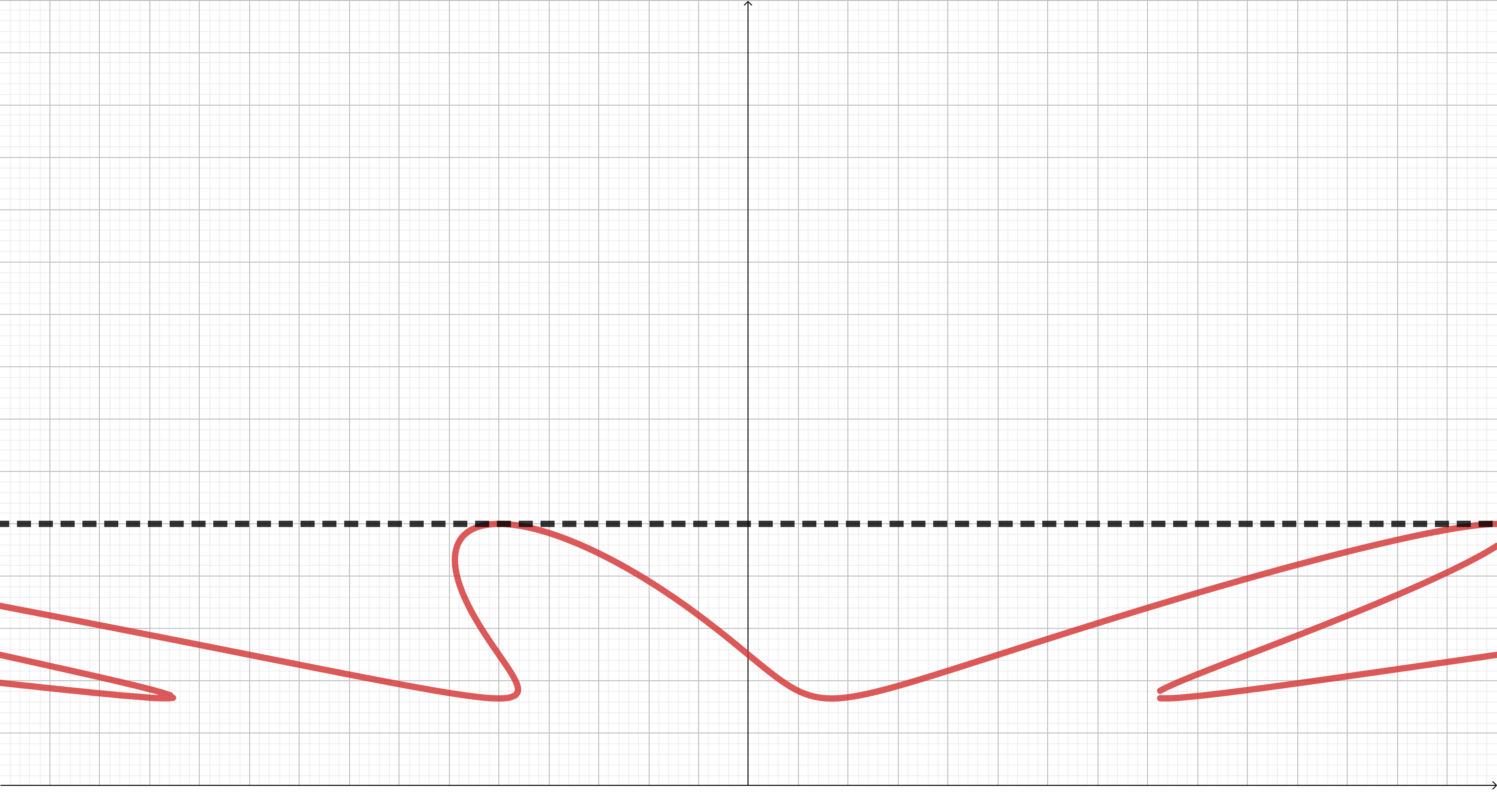}
		\captionof{figure}{Dual of a sine wave.}
		\label{fig:set-d5}
	\end{minipage}
\end{figure}

\section{The Radial Function Transformation} \label{sec:function-transform}
Recall that we defined the {\it upper radial function transformation} of some $f\colon \varSpace\rightarrow\extPos$ based on the radial set transformation as
\begin{align}
\radTransSup{f}(y) &= \sup\{v>0 \mid (y,v)\in\radTransSet{(\epi f)}\} \nonumber\\
&= \sup\{v>0 \mid v\cdot f(y/v) \leq 1\}. \label{eq:def-radTransSup}
\end{align}
This transformation essentially applies $\Gamma$ to the epigraph of $f$ and then interprets $\radTransSet{(\epi f)}$ as the hypograph of a new function.
Alternatively, interpreting $\radTransSet{(\hypo f)}$ as the epigraph of a new function gives the {\it lower radial function transformation} defined by
\begin{align}
\radTransInf{f}(y) &= \inf\{v>0 \mid (y,v)\in\radTransSet{(\hypo f)}\}\nonumber\\
&= \inf\{v>0 \mid v\cdot f(y/v) \geq 1\}. \label{eq:def-radTransInf}
\end{align}

Based on \eqref{eq:def-radTransSup} and \eqref{eq:def-radTransInf}, these transformations can alternatively be defined using the {\it perspective function} of $f$, which we denote by $f^p(y,v) = v\cdot f(y/v)$ for any $v>0$. It is immediate from this viewpoint that 
\begin{equation} \label{eq:strict-increase-condition}
\radTransSup{f}=\radTransInf{f}
\end{equation}
if and only if for every fixed $y\in\varSpace$, $v\mapsto f^p(y,v)$ is nondecreasing in $v$ and is strictly increasing whenever $f^p(y,v)$ is finite.

Having nondecreasing $f^p(y,\cdot)$ can be understood in terms of the intersection of rays with the epigraph or hypograph of $f$. The following lemma shows that if $f^p(y,\cdot)$ is nondecreasing, the ray $\{\lambda(y,1)\mid \lambda>0\}$ has $\lambda(y,1)$ lie in the hypograph for all $\lambda<\lambda_0$ and lie in the epigraph for all $\lambda>\lambda_0$ for some $\lambda_0\in\extPos$. Thus the hypograph of any such function is star-shaped with respect to the origin. 
\begin{lemma}\label{lem:ray-def-of-radial-functions}
	The following three conditions are equivalent:\\
	(i) all $y\in\varSpace$ have $f^p(y,\cdot)$ nondecreasing,\\
	(ii) all $(y,v)\in\epi f$ and $t\leq 1$ have $(y,v)/t\in\epi f$,\\ 
	(iii) all $(y,v)\in\hypo f$ and $t\geq 1$ have $(y,v)/t\in\hypo f$.
\end{lemma}
\begin{proof}
	First suppose $f^p(y,\cdot)$ is nondecreasing and consider any $(y,v)\in\epi f$ and $t\leq 1$. Then $t\cdot f(y/t)\leq f(y)\leq v$, and so $(y,v)/t\in\epi f$.
	Conversely, suppose some $t<t'$ has $f^p(y,t) > f^p(y,t')$. Then every $f(y/t')<\alpha<(t/t')\cdot f(y/t)$ must have $(y/t',\alpha)\in\epi f$. However, dividing this point by $t/t'\leq 1$ gives $(y/t,(t'/t)\alpha)\not\in\epi f$. Hence $(i) \iff (ii)$. Symmetric arguments show the equivalent hypograph condition as $(i) \iff (iii)$. \qed
\end{proof}

We say a function $f$ is {\it upper (lower) radial} whenever $f^p(y,\cdot)$ is nondecreasing and upper (lower) semicontinuous for all fixed $y\in\varSpace$.
If in addition $f^p(y,\cdot)$ is strictly increasing on its domain for every fixed $y\in \varSpace$, we say $f$ is {\it strictly upper (lower) radial}. 
The following theorem shows being upper (lower) radial is exactly the condition for the duality of the point and set transformations~\eqref{eq:radial-set-duality} to carry over to the upper (lower) radial function transformation.

\begin{theorem}\label{thm:radial-function-duality}
	A function $f$ is upper radial if and only if $\biradTransSup{f}=f.$\\
	Likewise\footnote{Throughout this manuscript, we claim mirrored results for the lower radial transformation in our theorems and propositions. We omit the proofs for these as they parallel those for the upper radial case.}, a function $f$ is lower radial if and only if $\biradTransInf{f}=f.$
\end{theorem}
\begin{proof}
	Observe that $(\radTransSup{f})^p(x,\cdot)$ is nondecreasing since it can be written as
	\begin{align}
	u\cdot\radTransSup{f}(x/u) &= u\cdot \sup\{v>0 \mid v\cdot f(x/vu) \leq 1\}\nonumber\\
	&= \sup\{w>0 \mid w\cdot f(x/w) \leq u\}.\nonumber
	\end{align}
	Then the twice radially transformed function equals the following infimum
	\begin{align}
	\biradTransSup{f}(x) &= \inf\{u>0 \mid u\cdot \radTransSup{f}(x/u) > 1\}\nonumber\\
	&= \inf\{u>0 \mid \sup\{w>0 \mid w\cdot f(x/w) \leq u\} > 1\}\nonumber\\
	&= \inf\{u>0 \mid \exists w>1 \text{\ s.t.\ } w\cdot f(x/w) \leq u\}\nonumber\\
	&= \inf\{w\cdot f(x/w) \mid w> 1\}.\nonumber
	\end{align}
	The claimed duality follows as $f(x)=\inf\{w\cdot f(x/w) \mid w> 1\}$ if and only if $w\mapsto w\cdot f(x/w)$ is nondecreasing and upper semicontinuous on $w>0$ for all $x\in\varSpace$. Verifying this fact is straightforward:
	
	For the forward direction, nondecreasing follows for any $0<w_1<w_2$ as $w_1 f(x/w_1) = w_1 \cdot \inf\{ w\cdot f(x/(w\times w_1)) \mid w>1\}\leq w_2\cdot f(x/w_2)$ (by plugging in $w=w_2/w_1>1$). Then upper semicontinuity at any $w_1>0$ follows immediately as
	$$\limsup_{w_2\nearrow w_1} w_2\cdot f(x/w_2) \leq w_1\cdot f(x/w_1)$$
	and
	$$\limsup_{w_2\searrow w_1} w_2\cdot f(x/w_2) \leq \inf\{ww_1\cdot f(x/ww_1) \mid w> 1\} = w_1\cdot f(x/w_1) $$
	where both inequalities follow from the nondecreasing property.
	
	The reverse direction follows as for any $x$, $\inf\{w\cdot f(x/w) \mid w> 1\}=\lim_{w\searrow 1}w\cdot f(x/w) = 1\cdot f(x/1)$ using nondecreasing for the first equality and upper semicontinuity for the second. \qed
\end{proof}

The radial duality among upper (or lower) radial functions is central to understanding the radial function transformations. In Section~\ref{subsec:characterizing-radial-functions}, we begin by characterizing when important classes of functions are upper or lower radial (i.e., semicontinuous, differentiable, convex, and concave functions). Then Section~\ref{subsec:closure-radial} shows being radial is preserved under many standard operations (i.e., conic combinations, linear compositions, minimums, and maximums).
Sections~\ref{subsec:semicontinuous-radial}, \ref{subsec:piecewise-linear-radial}, \ref{subsec:concave-convex-radial}, and~\ref{subsec:quasiconcave-convex-radial} consider the radial transformations of semicontinuous, piecewise linear, concave/convex, and quasiconcave/quasiconvex functions, respectively.
We conclude this section by giving several examples of radial function transformations in Section~\ref{subsec:radial-function-examples} 

\subsection{Characterizing Radial Functions}\label{subsec:characterizing-radial-functions}

\subsubsection{Radial Semicontinuous Functions}
Here we consider when an upper (or lower) semicontinuous function is upper (or lower) radial, and thus when our duality result holds.
Unlike Lemma~\ref{lem:ray-def-of-radial-functions} which focuses on rays from the origin, we give a necessary and sufficient condition based on proximal normal vectors of the function's hypograph (or epigraph).
We find it suffices to consider whether the origin lies below the tangent hyperplane induced by each proximal normal vector.
\begin{proposition}\label{prop:semicontinuity-radial-NS}
	An upper semicontinuous $f$ is upper radial if and only if all $(x,u)\in \hypo f$ and $(\zeta,\delta)\in N^P_{\hypo f}((x,u))$ satisfy
	$ (\zeta,\delta)^T(x,u)\geq 0.$\\
	Likewise, a lower semicontinuous $f$ is lower radial if and only if all $(x,u)\in \epi f$ and $(\zeta,\delta)\in N^P_{\epi f}((x,u))$ satisfy
	$ (\zeta,\delta)^T(x,u)\leq 0.$
\end{proposition}
\begin{proof}
	First suppose $f$ is upper radial and consider any $(x,u)\in \hypo f$ and $(\zeta,\delta)\in N^P_{\hypo f}((x,u))$. By Lemma~\ref{lem:ray-def-of-radial-functions}, $(x,u)/t \in\hypo f$ for all $t\geq 1$. Since $(x,u)\in \proj_{\hypo f}((x,u)+\epsilon(\zeta,\delta))$ for some $\epsilon>0$, all $t\geq 1$ satisfy
	$$\|(x,u)+\epsilon(\zeta,\delta) - (x,u)\|^2 \leq \|(x,u)+\epsilon(\zeta,\delta) - (x,u)/t\|^2.$$
	Simplifying this gives $$0 \leq (1-1/t)^2\|(x,u)\|^2 + 2\epsilon(1-1/t)(\zeta,\delta)^T(x,u),$$ and so taking $t\rightarrow 1$ verifies $(\zeta,\delta)^T(x,u)\geq 0$.
	
	Note that $f^p(y,\cdot)$ is upper semicontinuous by assumption.
	Now suppose $f^p(y,\cdot)$ is not nondecreasing. Then Lemma~\ref{lem:ray-def-of-radial-functions} guarantees some $(x,u)\in\hypo f$ and $\gamma>1$ has $(x,u)/\gamma \not\in \hypo f$. The assumed upper semicontinuity guarantees $\hypo f$ is closed, and thus for some $\epsilon>0$,
	$B((x,u)/\gamma, \epsilon)\cap \hypo f =\emptyset.$
	Hence the following supremum is well defined
	$$ \gamma' := \sup\{1<t\leq \gamma \mid B((x,u)/t, \epsilon/2)\cap \hypo f \neq\emptyset\}.$$
	Notice that $1<\gamma'<\gamma$. Further, $\interior B((x,u)/\gamma', \epsilon/2)\cap \hypo f=\emptyset$. Moreover, since $\hypo f$ is closed, some $(x',u') \in \hypo f$ lies on the boundary of this ball -- that is, $\|(x,u)/\gamma' - (x',u')\| = \epsilon/2$. Then $\hypo f$ at $(x',u')$ has the following proximal normal vector
	$$ (\zeta',\delta') := (x,u)/\gamma' - (x',u') \in N^P_{\hypo f}((x',u')).$$
	Since all $\gamma' < t < \gamma$ have $(x',u') \not\in  B((x,u)/t, \epsilon/2)$,
	\begin{align*}
	\left(\frac{\epsilon}{2}\right)^2 & \leq \left\|\frac{(x,u)}{t} - (x',u')\right\|^2\\
	& = \left\|\frac{\gamma'}{t}(\zeta',\delta') - \left(1 - \frac{\gamma'}{t}\right)(x',u')\right\|^2\\
	& = \left\|\frac{\gamma'}{t}(\zeta',\delta')\right\|^2 - 2\frac{\gamma'}{t}\left(1 - \frac{\gamma'}{t}\right)(\zeta',\delta')^T(x',u') + \left\|\left(1 - \frac{\gamma'}{t}\right)(x',u')\right\|^2\\
	& = \left(\frac{\gamma'}{t}\frac{\epsilon}{2}\right)^2 - 2\frac{\gamma'}{t}\left(1 - \frac{\gamma'}{t}\right)(\zeta',\delta')^T(x',u') + \left\|\left(1 - \frac{\gamma'}{t}\right)(x',u')\right\|^2.
	\end{align*}
	Rearrangement of this inequality gives
	$$ 2(\zeta',\delta')^T(x',u') \leq -\left(\frac{t}{\gamma'}+1\right)\left(\frac{\epsilon}{2}\right)^2 + \left(\frac{t}{\gamma'} - 1\right)\left\|(x',u')\right\|^2.$$
	Taking $t\rightarrow \gamma'$ shows $(\zeta',\delta')^T(x',u')\leq -\epsilon^2/2 <0$. \qed
\end{proof}


\subsubsection{Radial Differentiable Functions}
Here we specialize the previous result for semicontinuous functions to differentiable functions. We want to allow functions like the previously considered example $f(x)=\sqrt{1-\|x\|^2}$ (with value $0$ whenever $\|x\|\geq 1$) in our theory here. To this end, we say a function is {\it continuously differentiable} if $f$ is continuous on $\varSpace$ and $\nabla f(x)$ exists and is continuous on its effective domain $\dom f = \{x\in\varSpace \mid f(x)\in\RR_{++}\}$.

For any such function and $x\in\dom f$, if some nonzero $(\zeta,\delta)\in N^P_{\epi f}((x,u))$ exists, then $u=f(x)$ and $(\zeta,\delta)=\lambda(\nabla f(x),-1)$ for some $\lambda\geq0$. Further, for a dense subset of $\dom f$, the converse holds (which follows from the density theorem of proximal calculus~\cite[Theorem 1.3.1]{Clarke1998-nonsmoothanalysis}).
Then the continuity of $\nabla f$ and Proposition~\ref{prop:semicontinuity-radial-NS} imply the following condition is necessary and sufficient for $f$ to be radial\footnote{A continuous functions is (strictly) upper radial if and only if it is (strictly) lower radial. In such cases, we simply say the function is (strictly) radial as a shorthand.}.
\begin{proposition} \label{prop:differentiable-radial-NS}
	A continuously differentiable $f$ is radial if and only if for all $x\in\dom f$,
	$ (\nabla f(x),-1)^T(x,f(x)) \leq 0.$
\end{proposition}
This characterization can be alternatively derived by considering when the partial derivative $\frac{\partial}{\partial v}f^p(y,v) = f(y/v)-\nabla f(y/v)^T(y/v)$ is nonnegative. Based on this observation, having a positive derivative is sufficient to ensure $f^p(y,\cdot)$ is strictly increasing on its domain, and thus $\radTransSup{f}=\radTransInf{f}$ by~\eqref{eq:strict-increase-condition}.
\begin{proposition} \label{prop:differentiable-strict-increase-S}
	A continuously differentiable $f$ is strictly radial if for all $x\in\dom f$,
	$ (\nabla f(x),-1)^T(x,f(x)) < 0.$
\end{proposition}

\subsubsection{Radial Convex and Concave Functions}

Lastly we consider conditions for convex or concave functions to be upper or lower radial. For convex functions, the proximal subdifferential and convex subdifferential are equal giving the following characterization.
\begin{proposition}\label{prop:convex-radial-NS}
	A lower semicontinuous convex $f$ is lower radial if and only if all $(x,u)\in\epi f$ and $(\zeta,\delta)\in N^C_{\epi f}((x,u))$ have
	$(\zeta,\delta)^T(x,u)\leq 0.$
\end{proposition}

For concave functions, we find that it suffices to have points arbitrarily close to the origin with nonzero function value. As a result, every upper semicontinuous concave function can be translated to become upper radial.

\begin{proposition}\label{prop:concave-radial-NS}
	An upper semicontinuous concave $f$ has $f^p(y,\cdot)$ nondecreasing if and only if $0\in\closure\{x \mid f(x)>0\}$ or $f=0$.
\end{proposition}
\begin{proof}
	Trivially $f=0$ has $f^p(y,\cdot)$ nondecreasing and so we assume some $x'$ has $f(x')>0$.
	Then $f^p(y,\cdot)$ being nondecreasing implies all $t\geq 1$ have
	$f(x'/t)\geq f(x')/t>0.$
	Taking $t\rightarrow \infty$ gives a sequence verifying $0\in\closure\{y \mid f(y)>0\}$.
	
	Conversely, suppose $0\in\closure\{x \mid f(x)>0\}$ and consider any $(x,u)\in\hypo f$. Since $\hypo f$ is closed and convex and $(0,0) \in \closure\hypo f$, the line segment $((0,0),(x,u)]$ must lie in $\hypo f$. This is equivalent to $f^p(y,\cdot)$ being nondecreasing by Lemma~\ref{lem:ray-def-of-radial-functions}. \qed
\end{proof}
Furthermore, if the origin lies in the interior of $\{x\mid f(x)>0\}$ for any concave $f$, we find that $f^p(x,\cdot)$ is strictly increasing on its domain, and thus $\radTransSup{f}=\radTransInf{f}$ by~\eqref{eq:strict-increase-condition}. This condition can easily be attained for any concave $f$ whenever a point in the interior of the function's domain is known by translating it to the origin.
This directly corresponds to the setting assumed by Grimmer~\cite{Grimmer2017-radial-subgradient} and is equivalent to the conic setting assumed by Renegar~\cite{Renegar2016}.  
\begin{proposition}\label{prop:concave-strictly-radial-S}
	A concave $f$ has $f^p(y,\cdot)$ strictly increasing on its domain if $0\in\interior\{x \mid f(x)>0\}$.
\end{proposition}
\begin{proof}
	Consider any $y\in\varSpace$ and $0<v<v'$ with $v\cdot f(y/v)\in\RR_{++}$. Since $(y/v, f(y/v))\in\hypo f$, the concavity of $f$ ensures that the line segment $((0,0),(y/v, f(y/v)))$ lies in the interior of the hypograph of $f$. In particular,
	$(v/v')\cdot (y/v, f(y/v)) = (y/v', (v/v')\cdot f(y/v))\in\interior\hypo f$. Therefore $f(y/v') > (v/v')f(y/v)$ and so $f^p(y,v) < f^p(y,v')$. \qed
\end{proof}

\subsection{Closure of Radial Functions Under Common Operations}\label{subsec:closure-radial}
Building on our characterizations of when important classes of functions are upper or lower radial, here we show that this structure is preserved under many common operations.
The following result shows this is the case for any conic combination (that is, $\sum_{i=1}^k \lambda_i f_i(x)$ with each $\lambda_i> 0$), composition with linear maps, and taking finite minimums and maximums.
\begin{proposition}\label{prop:preserving-radial}
	For any pair of (strictly) upper radial functions $f_1,f_2$, the following functions are also (strictly) upper radial functions:\\
	(i) Positive Rescaling by $\lambda>0$: $\lambda \cdot f_1 $,\\
	(ii) Composition with a linear map $A \colon \varSpace' \rightarrow \varSpace$: $f_1\circ A$,\\
	(iii) Addition: $f_1+f_2$,\\
	(iv) Minimums: $\min\{f_1,f_2\}$,\\
	(v) Maximums: $\max\{f_1,f_2\}$.\\
	Likewise, these operations all preserve being (strictly) lower radial.
\end{proposition}
\begin{proof}
	Each of these operations preserves upper and lower semicontinuity and being nondecreasing (or strictly increasing). Consequently, they preserve being (strictly) upper or lower radial. \qed
\end{proof}
Note that (i) and (iii) above together give the claimed result for conic combinations.
For all of these operations except addition, we can give simple formulas for the resulting radial transformation, formalized in the following proposition.
\begin{proposition}\label{prop:radial-formulas}
	For any pair of functions $f_1,f_2$, the following identities hold for any $\lambda>0$ and linear $A \colon \varSpace' \rightarrow \varSpace$
	\begin{align*}
		\radTransSup{(\lambda \cdot f_1)}(y) &= \radTransSup{f_1}(\lambda y)/\lambda, \\
		\radTransSup{(f_1\circ A)} &= \radTransSup{f_1}\circ A,\\
		\radTransSup{\min\{f_1,f_2\}} &= \max\{\radTransSup{f_1},\radTransSup{f_2}\}.
	\end{align*}
	Further, if $f_1,f_2$ are upper radial, then
	\begin{align*}
		\radTransSup{\max\{f_1,f_2\}} &= \min\{\radTransSup{f_1},\radTransSup{f_2}\}.
	\end{align*}
	Likewise, for the lower radial transformation $\radTransInf{(\lambda \cdot f_1)}(y) = \radTransInf{(f_1)}(\lambda y)/\lambda$, $\radTransInf{(f_1\circ A)} = \radTransInf{(f_1)}\circ A$ and $\radTransInf{\max\{f_1,f_2\}} = \min\{\radTransInf{(f_1)},\radTransInf{(f_2)}\}$, and when given lower radial functions, $\radTransInf{\min\{f_1,f_2\}} = \max\{\radTransInf{(f_1)},\radTransInf{(f_2)}\}$.
\end{proposition}
\begin{proof}
	The results for positive rescaling by some $\lambda>0$, for composition with a linear map $A \colon \varSpace' \rightarrow \varSpace$, and for minimums follow immediately from the definition of our radial transformation as
	\begin{align*}
	\radTransSup{(\lambda\cdot f)}(y) &= \sup\{v>0 \mid \lambda v\cdot f(y/v)\leq 1\}\\
	&= \sup\{w>0 \mid w\cdot f(\lambda y/ w)\leq 1\}/\lambda
	= \radTransSup{f}(\lambda y)/\lambda,\\
	\radTransSup{(f\circ A)}(y) &= \sup\{v>0 \mid v\cdot f(Ay/v)\leq 1\}\\
	&= \radTransSup{f}(Ay),\\
	\radTransSup{\min\{f_1,f_2\}}(y) &= \sup\{v>0 \mid v\cdot \min\{f_1(y/v), f_2(y/v)\}\leq 1\}\\
	&= \max\limits_{i=1,2}\{\sup\{v>0 \mid v\cdot f_i(y/v)\leq 1\}\}
	= \max\{\radTransSup{f_1}(y), \radTransSup{f_2}(y)\}.
	\end{align*}
	The claimed formula for maximums follows as
	\begin{align*}
	\radTransSup{\max\{f_1,f_2\}}(y) &= \sup\{v>0 \mid v\cdot \max\{f_1(y/v), f_2(y/v)\}\leq 1\}\\
	&= \min\limits_{i=1,2}\{\sup\{v>0 \mid v\cdot f_i(y/v)\leq 1\}\}
	= \min\{\radTransSup{f_1}(y), \radTransSup{f_2}(y)\}
	\end{align*}
	where the second equality above relies on $v\cdot f_i(y/v)$ being nondecreasing. \qed
\end{proof}

From these simple operations, we can build up to more complex functions that preserve being upper radial. For example, consider the operations of taking the $k$th largest or smallest element out of a set of $n$ elements
\begin{align*}
	k\mhyphen\min\{x_1,\dots,x_n\} &:= x_{i_k} \text{ where } x_{i_1}\leq x_{i_2} \leq \dots \leq x_{i_n}\\
	k\mhyphen\max\{x_1,\dots,x_n\} &:= x_{i_{n-k+1}} \text{ where } x_{i_1}\leq x_{i_2} \leq \dots \leq x_{i_n}.
\end{align*}
and averaging the $k$ largest or smallest elements 
\begin{align*}
k\mhyphen\mathrm{minavg}\{x_1,\dots,x_n\} &:= \frac{1}{k}\sum^k_{j=1}x_{i_j} \text{ where } x_{i_1}\leq x_{i_2} \leq \dots \leq x_{i_n}\\
k\mhyphen\mathrm{maxavg}\{x_1,\dots,x_n\} &:= \frac{1}{k}\sum^k_{j=1}x_{i_{n-j+1}} \text{ where } x_{i_1}\leq x_{i_2} \leq \dots \leq x_{i_n}.
\end{align*}
\begin{corollary}
	For any (strictly) upper radial functions $f_1,\dots,f_n$, the functions $k\mhyphen\min\{f_i(x)\}$, $k\mhyphen\max\{f_i(x)\}$, $k\mhyphen\mathrm{minavg}\{f_i(x)\}$, and $k\mhyphen\mathrm{maxavg}\{f_i(x)\}$ are all (strictly) upper radial with
	$ \left(k\mhyphen\min\{f_i(x)\}\right)^\Gamma(y) = k\mhyphen\max\{f^\Gamma_i(y)\}.$
\end{corollary}
This follows from Propositions~\ref{prop:preserving-radial} and \ref{prop:radial-formulas} since these operations can be described as combinations of minimums, maximums, rescaling, and addition 
\begin{align*}
&k\mhyphen\min\{f_1(x),\dots,f_n(x)\} = \min\{\max\{f_i(x)\mid i\in S\}\mid S\subseteq\{1,\dots,n\}, |S|=k\},\\
&k\mhyphen\max\{f_1(x),\dots,f_n(x)\} = \max\{\min\{f_i(x)\mid i\in S\}\mid S\subseteq\{1,\dots,n\}, |S|=k\},\\
&k\mhyphen\mathrm{minavg}\{f_1(x),\dots,f_n(x)\} =\min\!\left\{\frac{1}{k}\sum_{i\in S} f_i(x)\mid S\subseteq\{1,\dots,n\}, |S|=k\right\}\!,\\
&k\mhyphen\mathrm{maxavg}\{f_1(x),\dots,f_n(x)\} =\max\!\left\{\frac{1}{k}\sum_{i\in S} f_i(x)\mid S\subseteq\{1,\dots,n\}, |S|=k\right\}\!.
\end{align*}

\subsection{Radial Transformation of Semicontinuous Functions} \label{subsec:semicontinuous-radial}
Next we turn our focus to understanding how various important families of functions behave under the radial transformation. Considering the transformation of upper and lower semicontinuous functions shows that these become lower and upper semicontinuous, respectively, when $f$ is appropriately radial.
\begin{proposition} \label{prop:semicontinuous-condition}
	For any lower semicontinuous, lower radial $f$, $\radTransSup{f}$ is upper semicontinuous.
	Likewise, for any upper semicontinuous, upper radial $f$, $\radTransInf{f}$ is lower semicontinuous.
\end{proposition}
\begin{proof}
	Consider any $y\in\varSpace$. Upper semicontinuity trivially holds at $y$ if $\radTransSup{f}(y)=\infty$. Now assume $\radTransSup{f}(y)<\infty$ and consider any $\gamma > \radTransSup{f}(y)$. Then $\gamma\cdot f(y/\gamma)>1$. From the lower semicontinuity of $f$, for some $\epsilon>0$, all $y'\in B(y,\epsilon)$ satisfy $\gamma\cdot f(y'/\gamma)> 1$. Therefore $\radTransSup{f}(y')\leq \gamma$. Taking the limit as $\gamma$ approaches $\radTransSup{f}(y)$ shows $\radTransSup{f}(y) = \limsup_{y'\rightarrow y} \radTransSup{f}(y')$. \qed
\end{proof}

These results with upper and lower semicontinuity reversed do not hold in general. For example, see the absolute value function in Figures~\ref{fig:func-p1}, which is continuous and both upper and lower radial but its upper transformation (Figure~\ref{fig:func-u1}) is only upper semicontinuous and its lower transformation (Figure~\ref{fig:func-l1}) is only lower semicontinuous. However, whenever $\radTransSup{f}=\radTransInf{f}$, the reversed propositions immediately hold. Thus, for strictly upper (lower) radial functions, upper semicontinuity and lower semicontinuity are dual to each other.

\begin{proposition} \label{prop:continuous-condition}
	For any $f$ with $f^p(y,\cdot)$ strictly increasing on its domain for every fixed $y\in\varSpace$,
	\begin{align*}
	f\text{\ upper\ semicontinuous} &\implies \radTransSup{f}= \radTransInf{f}\text{\ lower\ semicontinuous},\\
	f\text{\ lower\ semicontinuous} &\implies \radTransSup{f}=\radTransInf{f}\text{\ upper\ semicontinuous},\\
	f\text{\ continuous} &\implies \radTransSup{f}=\radTransInf{f}\text{\ continuous}.
	\end{align*}
\end{proposition}
\begin{proof}
	Since $\radTransSup{f}=\radTransInf{f}$ by~\eqref{eq:strict-increase-condition}, both directions of Proposition~\ref{prop:semicontinuous-condition} apply. \qed
\end{proof}

\subsection{Radial Transformation of Piecewise Linear Functions}\label{subsec:piecewise-linear-radial}
We say a function $f$ is {\it convex polyhedral} if $\epi f$ is the intersection of finitely many halfspaces and $\varSpace\times\RR_{++}$. Likewise, $f$ is {\it concave polyhedral} if $\hypo f$ is the intersection of finitely many halfspaces and $\varSpace\times\RR_{++}$. Recall Corollary~\ref{cor:polyhedral-sets} ensures polyhedral sets map to polyhedral sets under the radial set transformation. The following proposition shows how this property is mirrored by the radial function transformation on polyhedral functions.
\begin{proposition}\label{prop:piecewise-linear-conversion}
	If $f$ is convex polyhedral then $\radTransSup{f}$ is concave polyhedral.\\
	Likewise, if $f$ is concave polyhedral then $\radTransInf{f}$ is convex polyhedral.
\end{proposition}
\begin{proof}
	If $\epi f$ is polyhedral, then Corollary~\ref{cor:polyhedral-sets} implies $\radTransSet{(\epi f)}$ is also polyhedral. Since $\radTransSet{(\epi f)}$ is closed with respect to $\varSpace\times\RR_{++}$, the hypograph of $\radTransSup{f}$ can be written as
	$$\hypo \radTransSup{f} = \{(y,v) \mid \exists v'\geq v,\ (y,v')\in\radTransSet{(\epi f)}\}.$$
	Then Fourier-Motzkin elimination ensures $\hypo \radTransSup{f}$ is polyhedral. \qed
\end{proof}
Like the previous results on semicontinuity, the converses do not hold in general. However, whenever $f^p(y,\cdot)$ is strictly increasing on its domain for all $y\in\varSpace$, they immediately hold as $\radTransSup{f}=\radTransInf{f}$.


\subsection{Radial Transformation of Concave/Convex Functions}\label{subsec:concave-convex-radial}
Recall from Proposition~\ref{prop:convex-set-radial} that the radial set transformation preserves convexity. This structure carries over to the function setting where convex functions become concave and vice versa.
\begin{proposition}\label{prop:concave-conversion}
	If $f$ is concave then $\radTransSup{f}$ and $\radTransInf{f}$ are convex.\\
	Likewise, if $f$ is convex then $\radTransSup{f}$ and $\radTransInf{f}$ are concave.
\end{proposition}
\begin{proof}
	Note that the perspective function $f^p(y,v)$ is concave (convex) whenever $f$ is concave (convex)~\cite{Boyd-ConvexOptimization}. 
	Supposing $f$ is concave. Consider any $(y,v),(y',v')\in\epi \radTransSup{f}$ and $0\leq\lambda\leq 1$. Note all $t>v$ and $t'>v'$ have $t\cdot f(y/t)> 1$ and $t'\cdot f(y'/t')> 1$. Then the concavity of $f^p$ implies 
	\begin{align*}
	(\lambda&t + (1-\lambda)t')\cdot f\left(\frac{\lambda y + (1-\lambda) y'}{\lambda t + (1-\lambda)t'}\right) > 1.
	\end{align*}
	Thus $\radTransSup{f}(\lambda y + (1-\lambda) y') \leq \lambda v + (1-\lambda)v'$.
	
	Next consider any $(y,v),(y',v')\in\epi \radTransInf{f}$ and $0\leq\lambda\leq 1$. Note there must exist $t,t'$ near $\radTransInf{f}(y),\radTransInf{f}(y')$ with $t\cdot f(y/t)\geq 1$ and $t'\cdot f(y'/t')\geq 1$. Then the concavity of $f^p$ implies 
	\begin{align*}
	(\lambda&t + (1-\lambda)t')\cdot f\left(\frac{\lambda y + (1-\lambda) y'}{\lambda t + (1-\lambda)t'}\right) \geq 1.
	\end{align*}
	Thus $\radTransInf{f}(\lambda y + (1-\lambda) y') \leq \lambda \radTransInf{f}(y) + (1-\lambda)\radTransInf{f}(y')\leq \lambda v + (1-\lambda)v'$. \qed
\end{proof}

Thus the family of upper radial concave functions is dual to the family of upper radial convex functions.
This is particularly interesting because these families of functions are very different due to the symmetry breaking nature of working with the extended positive reals. Proposition~\ref{prop:concave-radial-NS} shows that any concave function can be translated to become radial, whereas no similar operation exists for convex functions.
This is a critical algorithmic insight since it allows us to take generic concave maximization problems and transform them into minimization problems that are both convex and upper radial. In the second part of this work, we will see that radially dual minimization problems are very structured, often being globally uniformly Lipschitz continuous, despite us starting with a quite generic maximization problem.


\subsection{Radial Transformation of Quasi-concave/-convex Functions}\label{subsec:quasiconcave-convex-radial}
Lastly we consider the generalization of concavity and convexity given by quasiconcavity and quasiconvexity.
We say a function is {\it quasiconcave (quasiconvex)} if its superlevel sets $\{x\in\varSpace\mid f(x)\geq z\}$ (sublevel sets $\{x\in\varSpace\mid f(x)\leq z\}$) are convex for all $z>0$. Similar to the previous section's results, we find that quasiconcave functions are dual to quasiconvex functions (although the additional condition that $f^p(y,\cdot)$ is nondecreasing is needed).

\begin{proposition}\label{prop:quasiconcave-conversion}
	If $f$ is quasiconcave, $\radTransSup{f}$ is quasiconvex. If in addition $f^p(y,\cdot)$ is nondecreasing, $\radTransInf{f}$ is quasiconvex.
	Likewise, if $f$ is quasiconvex, $\radTransInf{f}$ is quasiconcave. If in addition $f^p(y,\cdot)$ is nondecreasing, $\radTransSup{f}$ is quasiconcave.
\end{proposition}
\begin{proof}
	Suppose $f$ is quasiconcave and fix any level $z\in\RR_{++}$. Consider any $0\leq\lambda\leq 1$ and $y,y'\in\varSpace$ with $\radTransSup{f}(y)\leq z$ and $\radTransSup{f}(y')\leq z$. First, we consider the upper radial transformation.
	Note all $\gamma> z$ have $\gamma\cdot f(y/\gamma)> 1$ and $\gamma\cdot f(y'/\gamma)> 1$. Then the quasiconcavity of $f$ implies 
	\begin{align*}
	\gamma\cdot f\left(\frac{\lambda y + (1-\lambda) y'}{\gamma}\right) & \geq \gamma\cdot \min\left\{f\left(y/\gamma\right), f\left(y'/\gamma\right)\right\}\\
	& = \min\left\{\gamma\cdot f\left(y/\gamma\right), \gamma\cdot f\left(y'/\gamma\right)\right\} > 1.
	\end{align*}
	Thus $\radTransSup{f}(\lambda x + (1-\lambda) y) \leq z$ since this holds for every $\gamma>z$.
	
	Now we consider the lower radial transformation and further assume $f^p(y,\cdot)$ is nondecreasing. Note all $\gamma>z$ must have $\gamma\cdot f(y/\gamma)\geq 1$ and $\gamma\cdot f(y'/\gamma)\geq 1$. Then the quasiconvexity of $f$ implies 
	\begin{align*}
	\gamma\cdot f\left(\frac{\lambda y + (1-\lambda) y'}{\gamma}\right) & \geq \gamma\cdot \min\left\{f\left(y/\gamma\right), f\left(y'/\gamma\right)\right\}\\
	& = \min\left\{\gamma\cdot f\left(y/\gamma\right), \gamma\cdot f\left(y'/\gamma\right)\right\} \geq 1.
	\end{align*}
	Thus $\radTransInf{f}(\lambda y + (1-\lambda) y') \leq z$ since this holds for every $\gamma>z$. \qed
\end{proof}

\subsection{Examples and Pictures}\label{subsec:radial-function-examples}
In Figures~\ref{fig:func-p1} through~\ref{fig:func-ul5}, we give a number of examples of radial function transformations. As done in our illustrations of the radial set transformation, each figure includes the horizontal line $\{(x,1)\mid x\in\RR\}$ for reference.

Figures~\ref{fig:func-p1},~\ref{fig:func-u1}, and~\ref{fig:func-l1} show the absolute value function and its transformations of
$$\radTransSup{|\cdot|}(y) = \begin{cases} +\infty & \text{if } |y|\leq 1\\ 0 &\text{otherwise,}\end{cases} \quad \mathrm{and} \quad \radTransInf{|\cdot|}(y) = \begin{cases} +\infty & \text{if } |y|< 1\\ 0 &\text{otherwise}\end{cases} \ .$$
Note $|x|$ is both upper and lower radial, but not strictly. As a result, although $\radTransSup{|\cdot|}$ and $\radTransInf{|\cdot|}$ are not equal, both are dual to $|x|$ under their respective transformations.

Figures~\ref{fig:func-p2} and~\ref{fig:func-ul2} show the strictly radial, concave function $\sqrt{1-x^2}$ (with the value at all $x^2>1$ set to $0$) and its transformation $\sqrt{1+y^2}$. Notice the transformed function is convex as guaranteed by Proposition~\ref{prop:concave-conversion}.

Figures~\ref{fig:func-p3} and~\ref{fig:func-ul3} show the strictly radial function $e^{-|x|}+1/2$ and its transformation. This function is not concave but is quasiconcave and so, as guaranteed by Propositions~\ref{prop:concave-conversion} and~\ref{prop:quasiconcave-conversion}, its radial transformation is not convex but is quasiconvex.

Lastly, Figures~\ref{fig:func-p4},~\ref{fig:func-u4}, and~\ref{fig:func-ul5} continue the example of transforming the quadratic $(x+1)^2+1/2$ which was used in Figures~\ref{fig:set-p4} and~\ref{fig:set-d4} to illustrate the radial set transformation. Notice that this quadratic is not upper radial as its epigraph transforms into an ellipsoid-like shape rather than the hypograph of another function. Hence its upper radial transformation is not dual to the original quadratic. Figure~\ref{fig:func-ul5} shows the result of applying the upper radial transformation to this quadratic twice. In this case, the functions in Figures~\ref{fig:func-u4} and~\ref{fig:func-ul5} are upper radial and thus radially dual. There $h^{\Gamma\Gamma}$ could be viewed as a ``radial envelope'' or ``star closure'' of $h$. Exploration of this property may be of interest to future work.
\begin{figure}
	\centering
	\begin{minipage}{.30\textwidth}
		\centering
		\includegraphics[width=.98\linewidth]{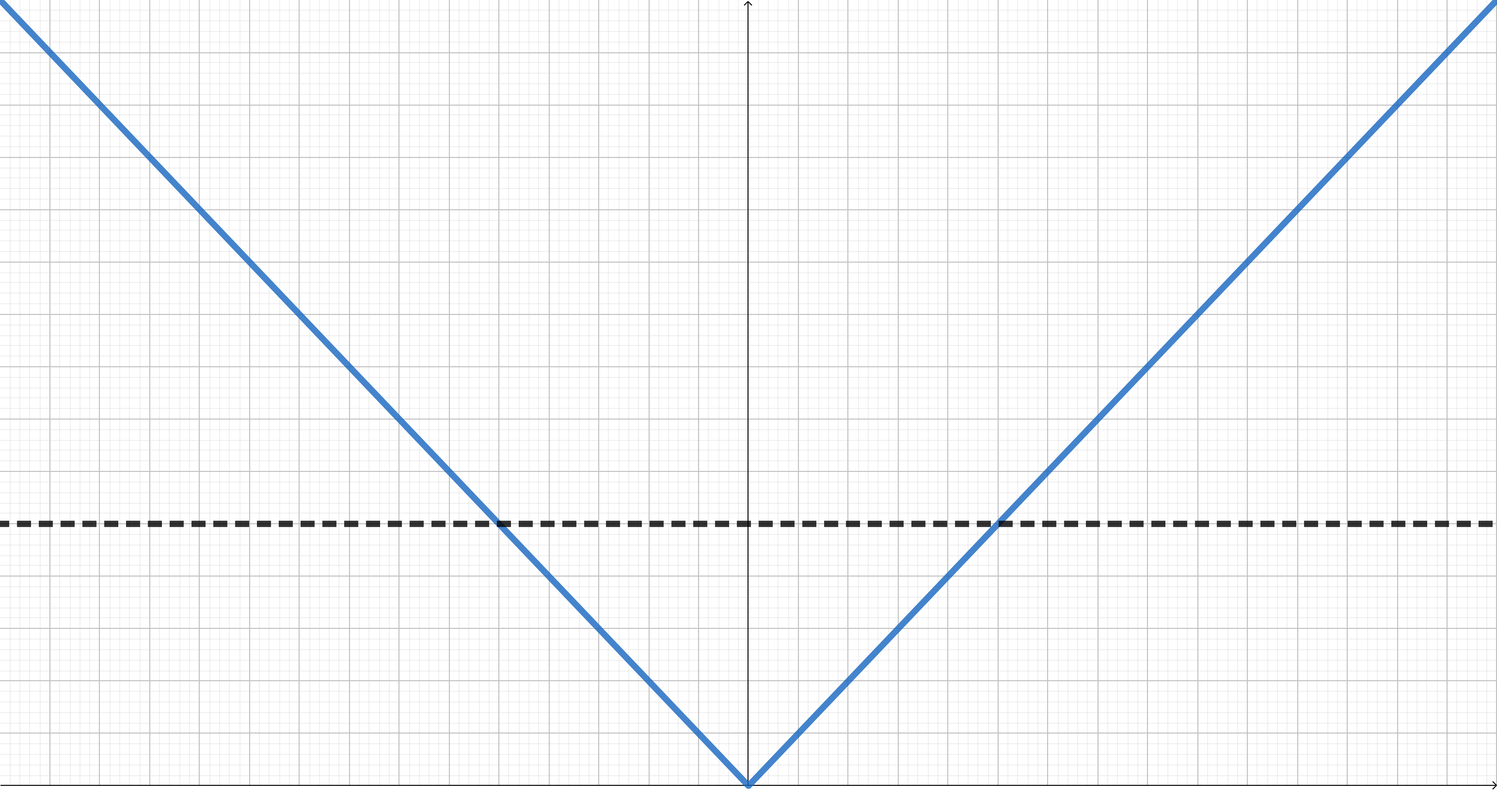}
		\captionof{figure}{$|x|$}
		\label{fig:func-p1}
	\end{minipage}%
	$\iff$
	\begin{minipage}{.30\textwidth}
		\centering
		\includegraphics[width=.98\linewidth]{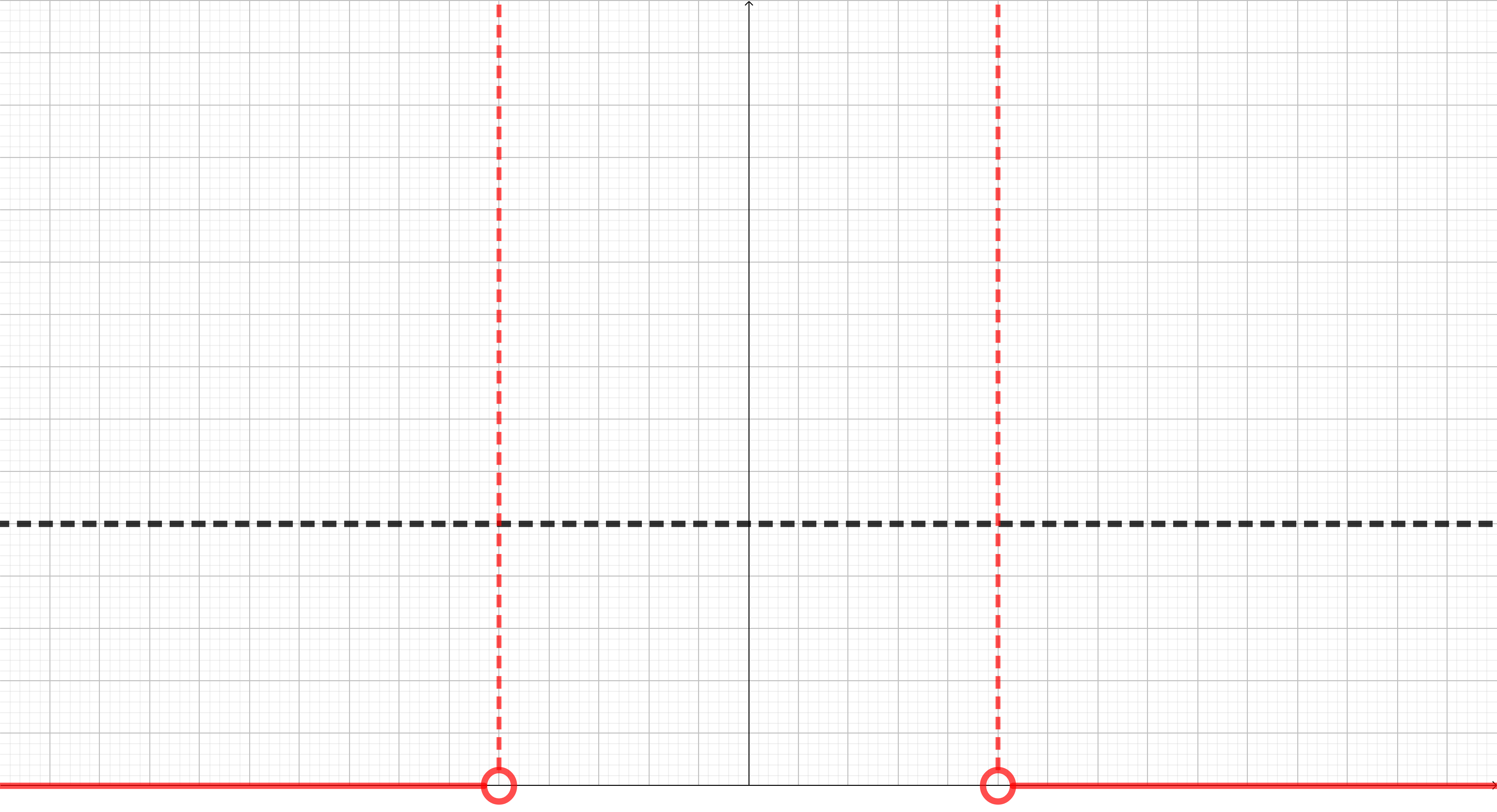}
		\captionof{figure}{$\radTransSup{|\cdot|}(y)$}
		\label{fig:func-u1}
	\end{minipage}
	\begin{minipage}{.30\textwidth}
		\centering
		\includegraphics[width=.98\linewidth]{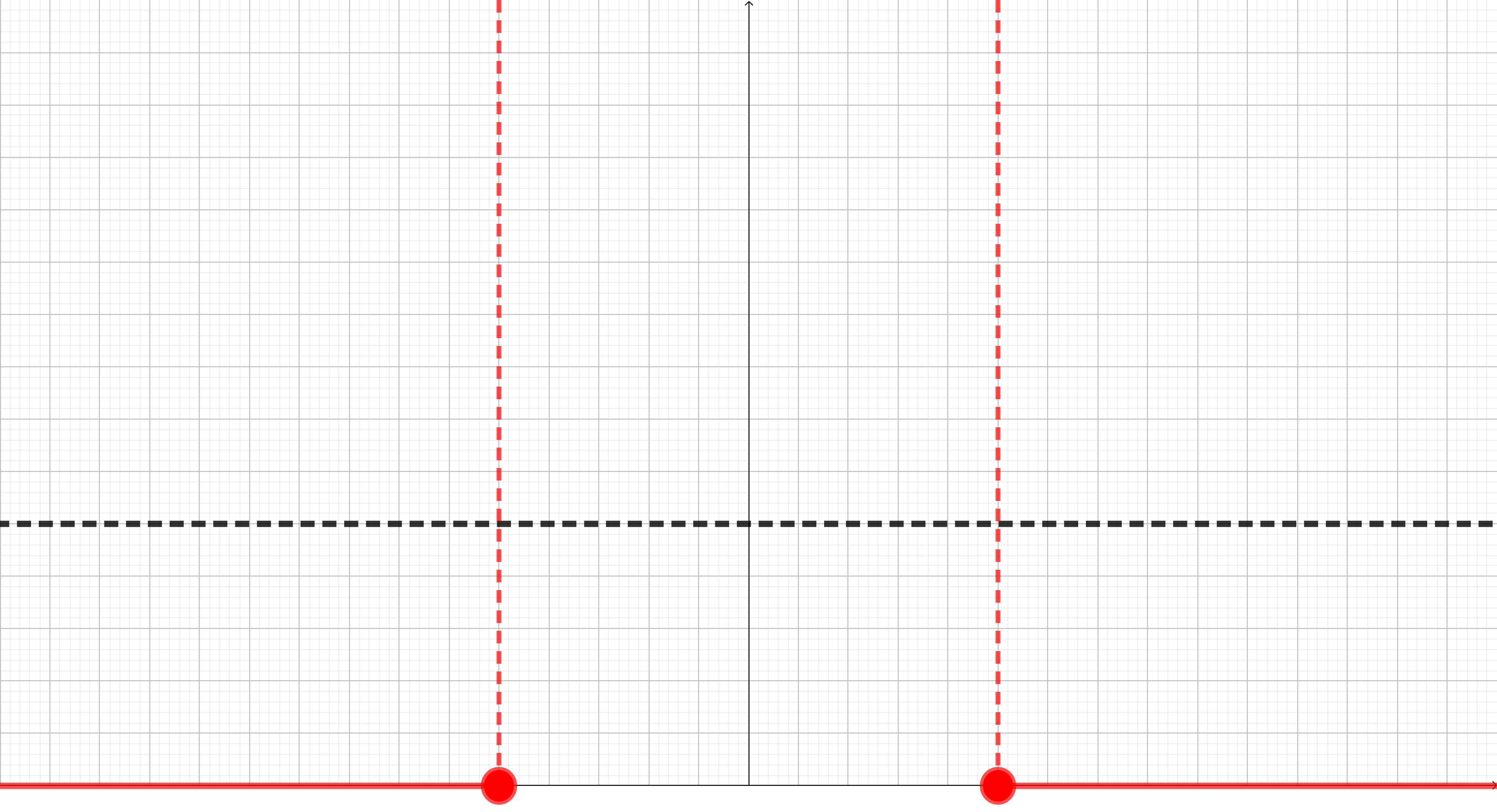}
		\captionof{figure}{$\radTransInf{|\cdot|}(y)$}
		\label{fig:func-l1}
	\end{minipage}
	
	\begin{minipage}{.4\textwidth}
		\centering
		\includegraphics[width=.98\linewidth]{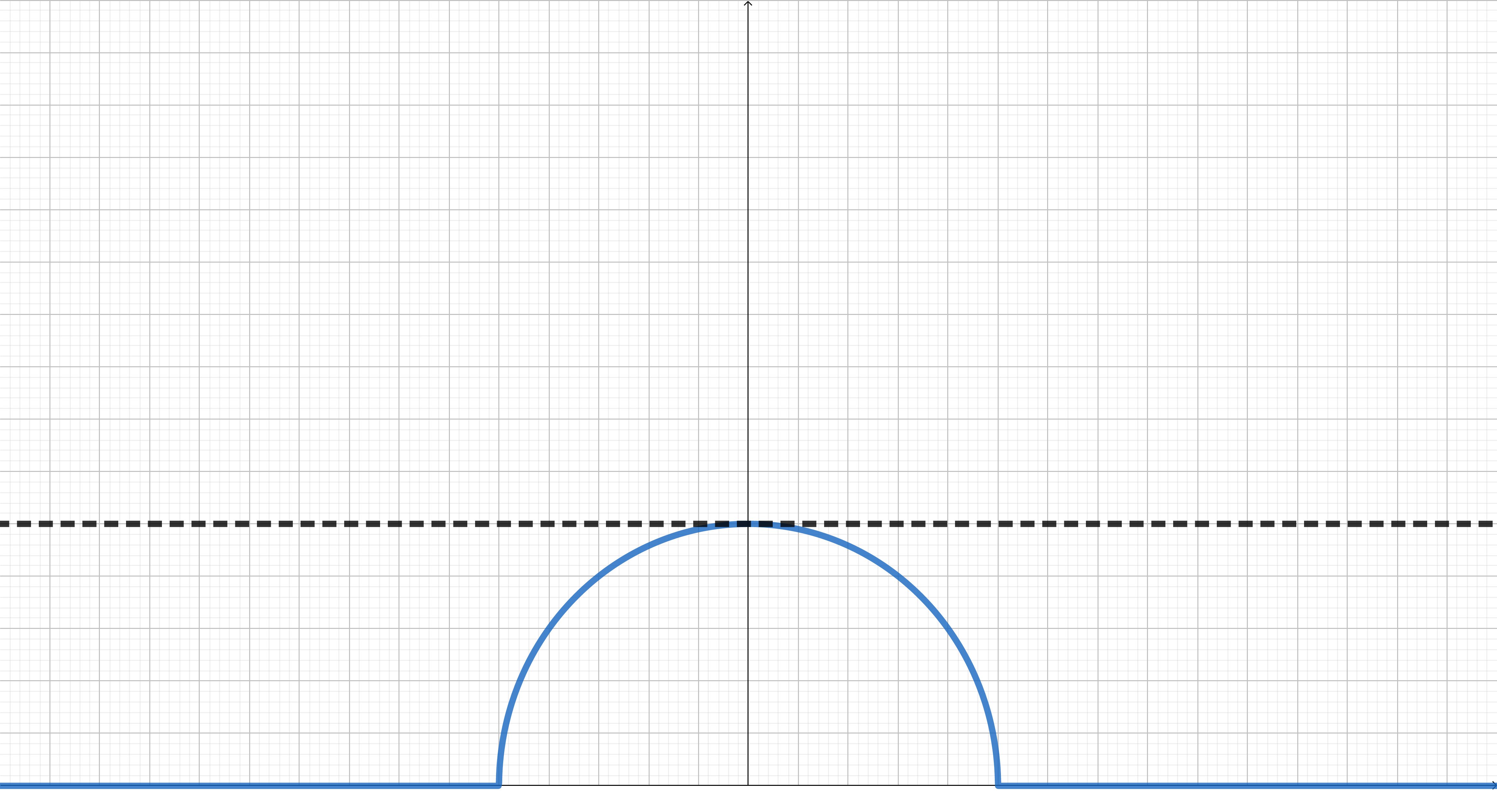}
		\captionof{figure}{$f(x)=\sqrt{1-x^2}$}
		\label{fig:func-p2}
	\end{minipage}%
	$\iff$
	\begin{minipage}{.4\textwidth}
		\centering
		\includegraphics[width=.98\linewidth]{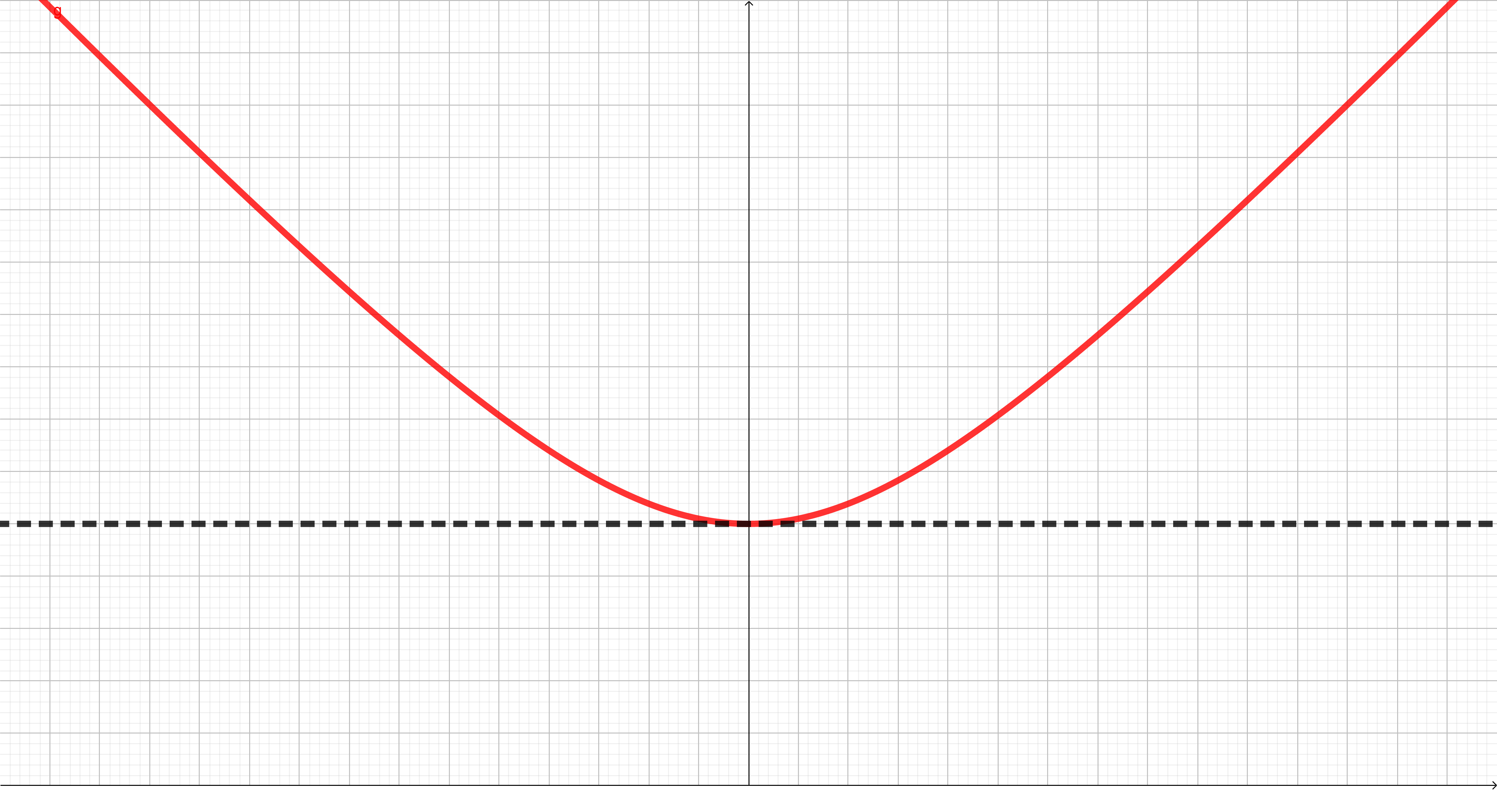}
		\captionof{figure}{$\radTransSup{f}(y)=\radTransInf{f}(y)=\sqrt{1+y^2}$}
		\label{fig:func-ul2}
	\end{minipage}
	
	\begin{minipage}{.4\textwidth}
		\centering
		\includegraphics[width=.98\linewidth]{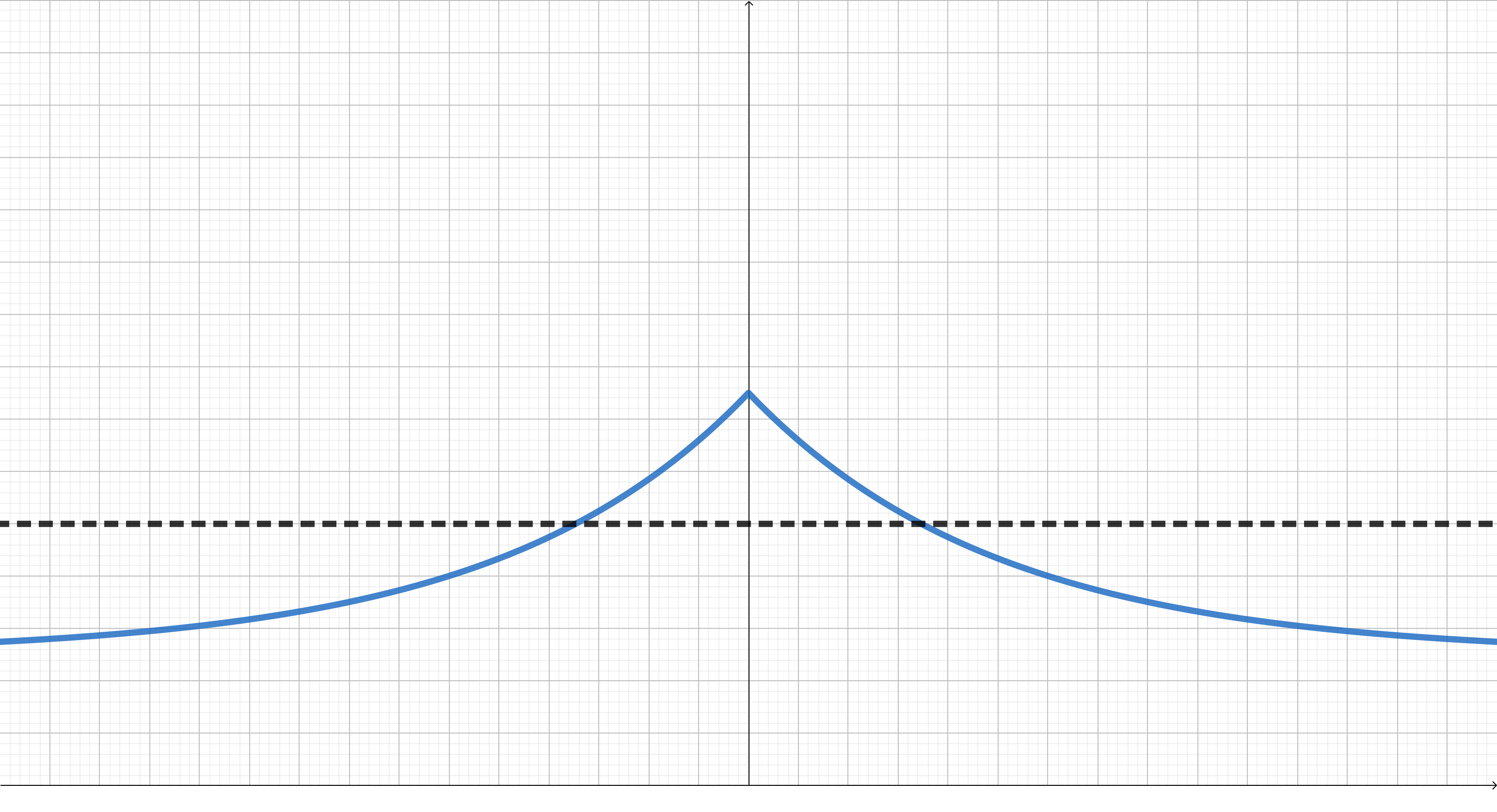}
		\captionof{figure}{$g(x)=e^{-|x|}+1/2$}
		\label{fig:func-p3}
	\end{minipage}%
	$\iff$
	\begin{minipage}{.4\textwidth}
		\centering
		\includegraphics[width=.98\linewidth]{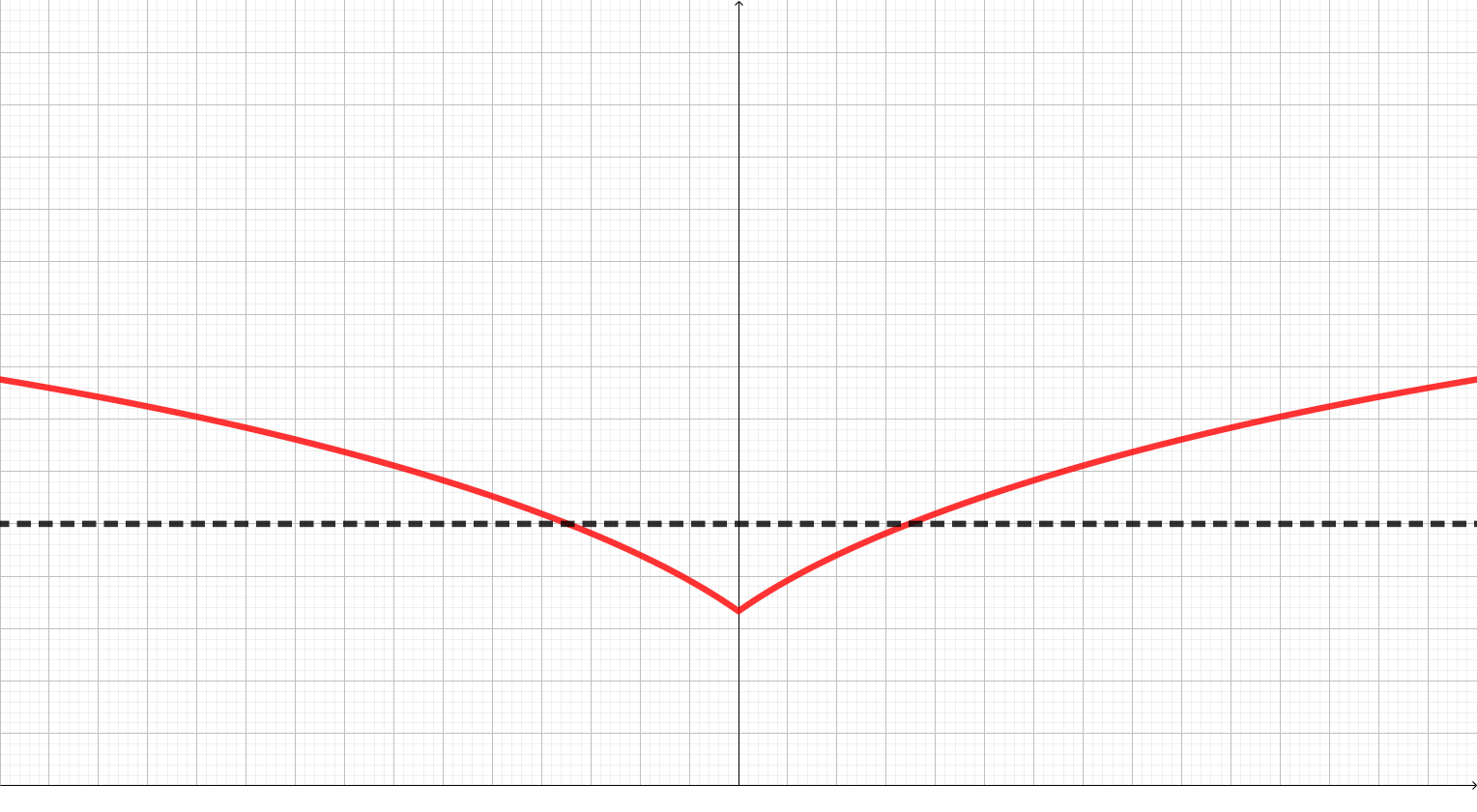}
		\captionof{figure}{$\radTransSup{g}(y)=\radTransInf{g}(y)$}
		\label{fig:func-ul3}
	\end{minipage}
	
	\begin{minipage}{.4\textwidth}
		\centering
		\includegraphics[width=.98\linewidth]{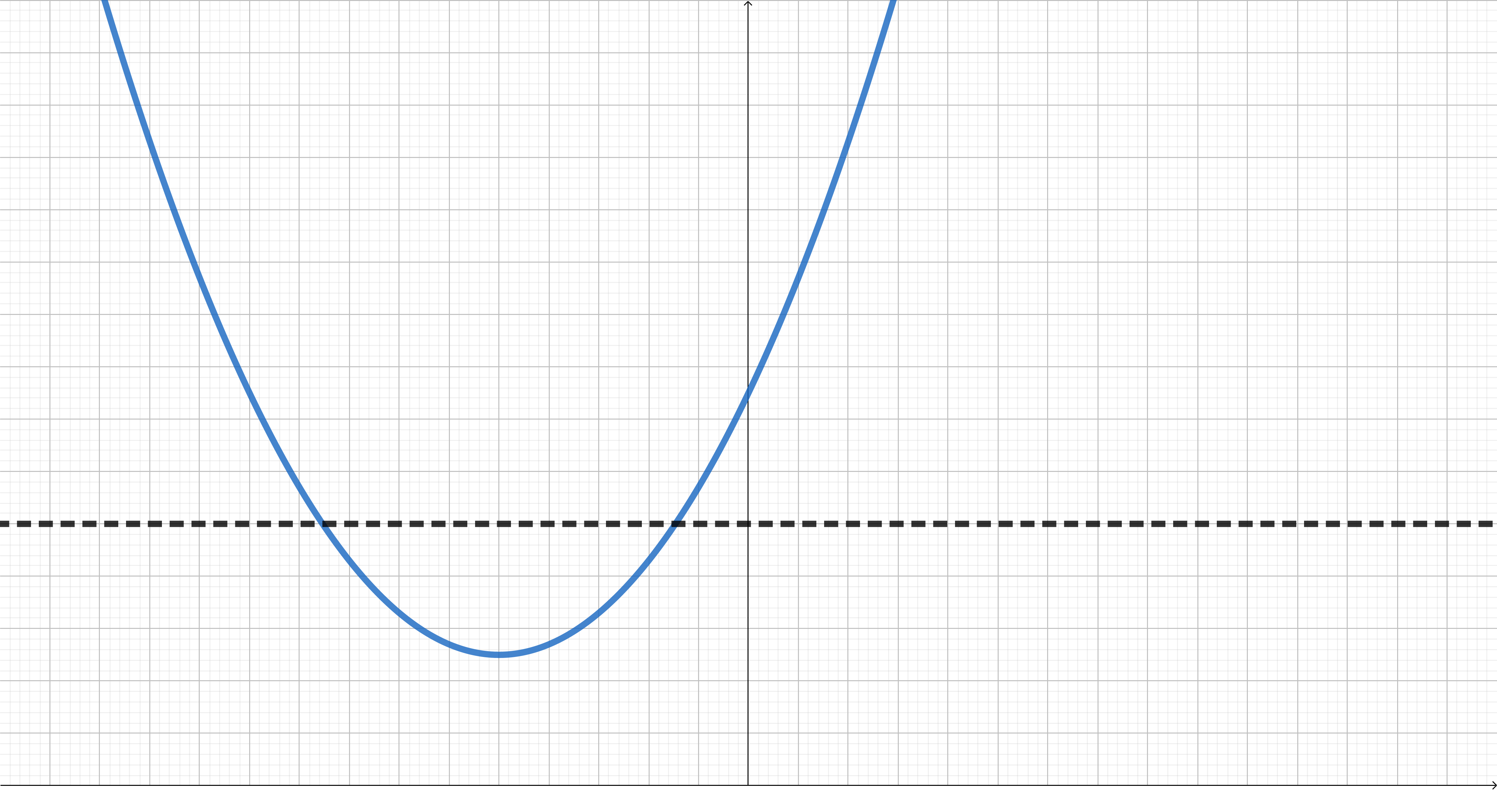}
		\captionof{figure}{$h(x)=(x+1)^2+1/2$}
		\label{fig:func-p4}
	\end{minipage}%
	$\implies$
	\begin{minipage}{.4\textwidth}
		\centering
		\includegraphics[width=.98\linewidth]{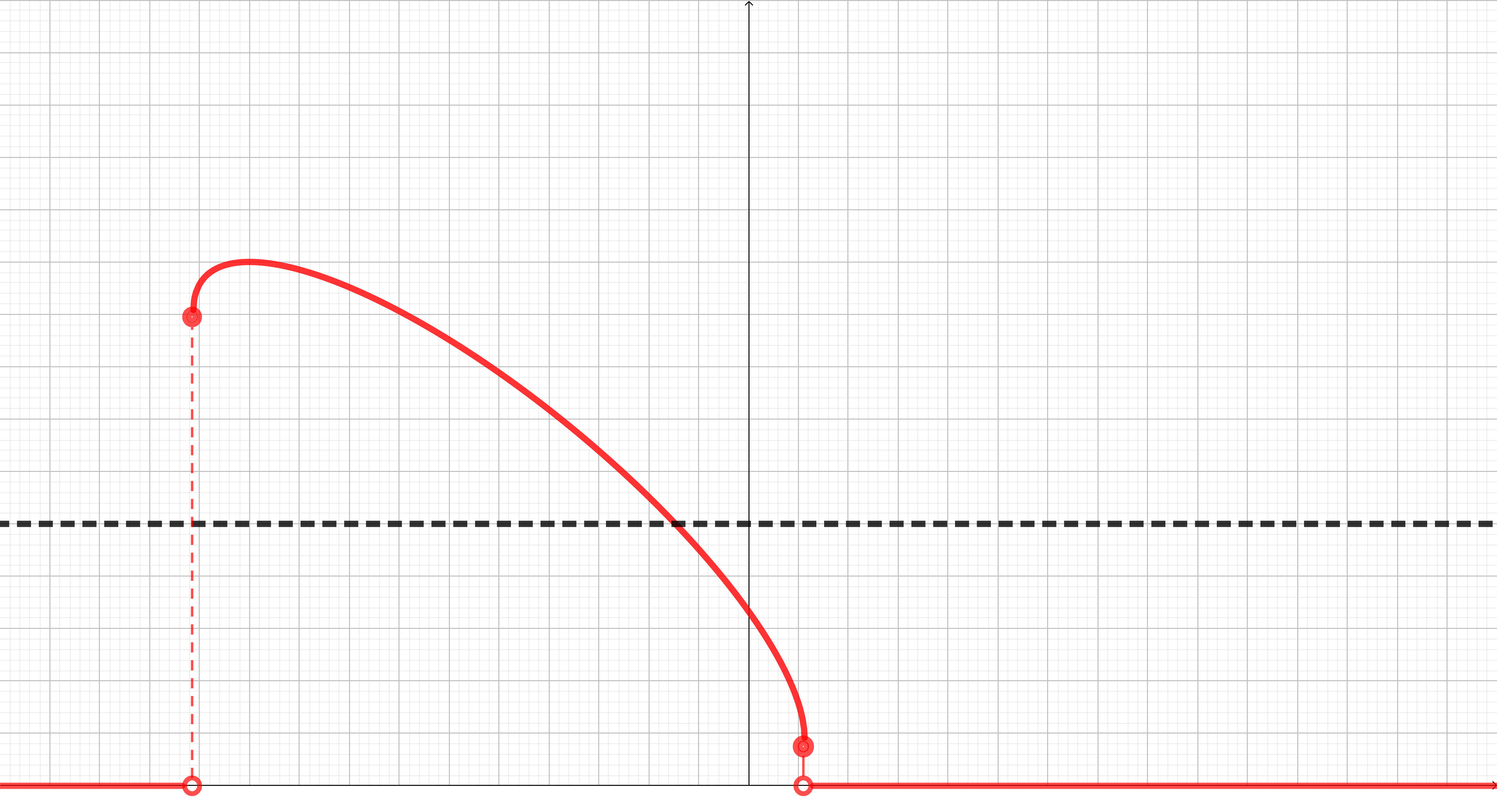}
		\captionof{figure}{$\radTransSup{h}(y)$}
		\label{fig:func-u4}
	\end{minipage}
	\begin{minipage}{.46\textwidth}
		\hfill $\iff$
	\end{minipage}%
	\begin{minipage}{.4\textwidth}
		\centering
		\includegraphics[width=.98\linewidth]{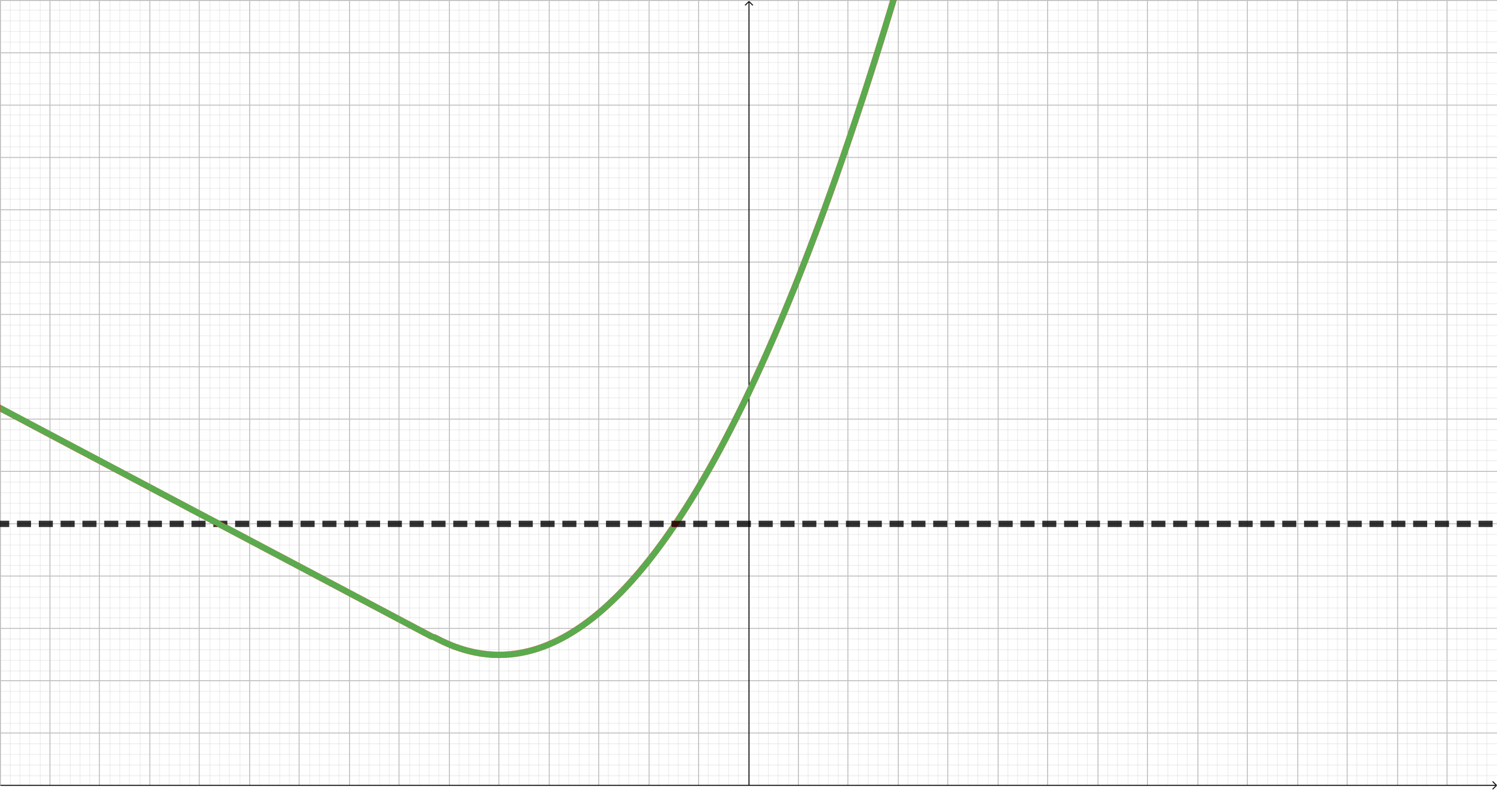}
		\captionof{figure}{$\biradTransSup{h}(x)$}
		\label{fig:func-ul5}
	\end{minipage}
\end{figure}

\section{Optimization Based on Radial Transformations} \label{sec:optimization}
Here we develop the necessary machinery to propose and analyze optimization methods based on the radial transformation. For any appropriately radial function, formulas for the convex and proximal subdifferentials and supdifferentials of its radial transformation are given in Section~\ref{subsec:subdifferentials}. Further, assuming $f$ is sufficiently differentiable, Section~\ref{subsec:gradients} characterizes the gradients and Hessians of its radial transformations.
In Section~\ref{subsec:optimal}, we relate the optimal points (minimizers and maximizers) and stationary points of a function and its radial transformations. These calculus and optimality relations form the foundations of relating the pair of radially dual optimization problems~\eqref{eq:base-problem} and~\eqref{eq:radial-problem}.

\subsection{Convex and Proximal Subgradients and Supgradients}\label{subsec:subdifferentials}
To understand the convex and proximal subdifferentials under radial function transformations, we leverage Propositions~\ref{prop:convex-normals} and~\ref{prop:proximal-normals} which described normal vectors under the radial set transformation.
The following lemma relates the epigraph and hypograph of a radially transformed function to those of $f$. 
\begin{lemma}\label{lem:hypo-epi-inclusion}
	For any upper radial $f$, $ \epi\radTransSup{f} \subseteq \Gamma(\hypo f)$.\\
	Likewise, for any lower radial $f$, $\hypo\radTransInf{f} \subseteq \Gamma(\epi f)$.\\
	If $f^p(y,\cdot)$ is strictly increasing on its domain, equality holds in both cases.
\end{lemma}
\begin{proof}
	Noting that $\radTransSup{f}(y)\leq v \implies v\cdot f(y/v)\geq 1$ for upper radial $f$, 
	\begin{align*}
	\Gamma(\hypo f) &= \left\{\frac{(x,1)}{u} \mid f(x)\geq u\right\}\\
	&= \left\{(y,v) \mid f(y/v)\geq 1/v\right\}\\
	&= \left\{(y,v) \mid v\cdot f(y/v)\geq 1\right\}\\
	&\supseteq \left\{(y,v) \mid f^\Gamma(y)\geq v\right\} = \epi\radTransSup{f}.
	\end{align*}
	Equality holds when $f$ is strictly upper radial as $\radTransSup{f}(y)\leq v \iff v\cdot f(y/v)\geq 1$. \qed
\end{proof}

In light of Lemma~\ref{lem:hypo-epi-inclusion}, we can immediately apply our results on normal vectors under the radial set transformation to understand differentials under the function transformation. The following pair of propositions do this for the convex and proximal subdifferential and supdifferential. A very similar result restricted to convex functions was previously given by~\cite[Corollary 3.4]{Artstein2017}. It is natural that our formulas match theirs as nonconvexities have little effect on local objects like (sub)gradients.
\begin{proposition} \label{prop:convex-subgradients}
	For any strictly upper radial $f$,
	$$ \partial_C\radTransSup{f}(y) = \left\{\frac{\zeta}{(\zeta,\delta)^T(x,u)} \mid \begin{bmatrix} \zeta \\ \delta\end{bmatrix}\in N^C_{\hypo f}((x,u)),\ (\zeta,\delta)^T(x,u)>0\right\}$$
	where $(x,u) = \Gamma(y,\radTransSup{f}(y))$. Likewise, for any strictly lower radial $f$,
	$$ \partial^C\radTransInf{f}(y) = \left\{\frac{\zeta}{(\zeta,\delta)^T(x,u)} \mid \begin{bmatrix} \zeta \\ \delta\end{bmatrix}\in N^C_{\epi f}((x,u)),\ (\zeta,\delta)^T(x,u)<0\right\}.$$
\end{proposition}
\begin{proof}
	Recall that Proposition~\ref{prop:convex-normals} characterized the convex normal vectors of the radial transformation of a set in terms of the original set. Since the assumed strict increase ensures equality holds in Lemma~\ref{lem:hypo-epi-inclusion}, this applies to the epigraph and hypograph of $\radTransSup{f}$ and $\radTransInf{f}$, respectively. Thus when $f$ is strictly upper radial
	\begin{equation}\label{eq:upper-radial-convex-normal}
	N^C_{\epi \radTransSup{f}}((y,v)) = \left\{\begin{bmatrix} \zeta \\ -(\zeta,\delta)^T(x,u)\end{bmatrix} \mid \begin{bmatrix} \zeta \\ \delta\end{bmatrix}\in N^C_{\hypo f}((x,u))\right\}
	\end{equation}
	and when $f$ is strictly lower radial
	\begin{equation}\label{eq:lower-radial-convex-normal}
	N^C_{\hypo \radTransInf{f}}((y,v)) = \left\{\begin{bmatrix} \zeta \\ -(\zeta,\delta)^T(x,u)\end{bmatrix} \mid \begin{bmatrix} \zeta \\ \delta\end{bmatrix}\in N^C_{\epi f}((x,u))\right\}.
	\end{equation}
	Then the claimed sub(sup)gradient formulas follow by definition. \qed
\end{proof}
\begin{proposition} \label{prop:proximal-subgradients}
	For any strictly upper radial $f$,
	$$ \partial_P\radTransSup{f}(y) = \left\{\frac{\zeta}{(\zeta,\delta)^T(x,u)} \mid \begin{bmatrix} \zeta \\ \delta\end{bmatrix}\in N^P_{\hypo f}((x,u)),\ (\zeta,\delta)^T(x,u)>0\right\}$$
	where $(x,u) = \Gamma(y,\radTransSup{f}(y))$. Likewise, for any strictly lower radial $f$,
	$$ \partial^P\radTransInf{f}(y) = \left\{\frac{\zeta}{(\zeta,\delta)^T(x,u)} \mid \begin{bmatrix} \zeta \\ \delta\end{bmatrix}\in N^P_{\epi f}((x,u)),\ (\zeta,\delta)^T(x,u)<0\right\}.$$
\end{proposition}
\begin{proof}
	Using Proposition~\ref{prop:proximal-normals} to describe the proximal normal of $\epi f^\Gamma$, this follows from the same reasoning as the proof of Proposition~\ref{prop:convex-subgradients}. \qed
\end{proof}

\subsection{Gradients and Hessians for Differentiable Functions} \label{subsec:gradients}
Here we narrow our focus to consider differentiable functions under the radial transformation. Whenever $\radTransSup{f}$ is differentiable, a formula for its gradient follows from the subgradient formula in Proposition~\ref{prop:proximal-subgradients}. To establish when  $\radTransSup{f}$ is differentiable, we show that being $k$ times continuously differentiable (or analytic) is preserved under the radial transformation for appropriate functions. Lastly, we give a formula for the Hessian of the radial transformation of any appropriate twice differentiable function. 

As a first step, we give a simple bijection between the graphs (and thus domains) of a function $f$ and its radial transformation $\radTransSup{f}$ whenever $f$ is continuous and strictly radial.
\begin{lemma}\label{lem:graph-bijection}
	If either $f$ or $f^\Gamma$ is continuous and strictly radial,
	$ \graph \radTransSup{f} = \Gamma(\graph f).$
	Hence, if $y\in\dom\radTransSup{f}$ then $y/\radTransSup{f}(y) \in\dom f$.
\end{lemma}
\begin{proof}
	If $f$ is continuous and strictly radial, then noting that $\graph f = \epi f \cap \hypo f$, we have
	\begin{align*}
	\graph \radTransSup{f} = \epi \radTransSup{f} \cap \hypo \radTransSup{f}
	= \Gamma(\epi f) \cap \Gamma(\hypo f)
	= \Gamma(\epi f \cap \hypo f)
	\end{align*}
	where the second equality follows from Lemma~\ref{lem:hypo-epi-inclusion} and the third from~\eqref{eq:radial-intersection}. If instead $f^\Gamma$ is continuous and strictly radial, the above reasoning implies $\graph f^{\Gamma\Gamma}=\Gamma(\graph f^{\Gamma})$. This is equivalent to $ \graph \radTransSup{f} = \Gamma(\graph f)$ by Theorem~\ref{thm:radial-function-duality}.\qed
\end{proof}

This lemma lets us view the graph of the radial transformation as the relation $\Gamma(\graph f)$. Applying the implicit function theorem to this relation shows differentiability is preserved under the transformation for appropriate functions. Then leveraging the previous section's results on the proximal subdifferential gives a formula for the gradient of the radial transformation. A similar result for convex functions was previously given by~\cite[Lemma 3.5]{Artstein2017}.
\begin{proposition}\label{prop:differentiable-preserved}
	Consider any continuous, strictly radial $f$ with points $(x,f(x))=\Gamma(y,\radTransSup{f}(y))$. Then $f$ is $k$ times continuously differentiable (or analytic) around $x$ with 
	$$(\nabla f(x),-1)^T(x,f(x))<0$$ if and only if $\radTransSup{f}=\radTransInf{f}$ is $k$ times continuously differentiable (or analytic) around $y$ with $$(\nabla \radTransSup{f}(y),-1)^T(y,\radTransSup{f}(y))<0$$
	where
	$ \nabla \radTransSup{f}(y) = \nabla f(x)/(\nabla f(x),-1)^T(x,f(x)).$
\end{proposition}
\begin{proof}
	It suffices to only show the forward direction as Theorem~\ref{thm:radial-function-duality} will then imply the reverse direction. Define the following $k$ times continuously differentiable (or analytic) function
	$ F(y',v')= v'\cdot f(y'/v') -1.$
	Then from Lemma~\ref{lem:graph-bijection}, $\graph \radTransSup{f} = \{(y',v')\in\varSpace\times\RR_{++} \mid F(y',v')=0\}$. 
	Noting
	$$\frac{\partial}{\partial v} F(y,\radTransSup{f}(y)) = f(y/\radTransSup{f}(y)) - \nabla f(y/\radTransSup{f}(y))^T(y/\radTransSup{f}(y)),$$
	we find $\frac{\partial}{\partial v} F(y,\radTransSup{f}(y)) = f(x) - \nabla f(x)^Tx >0$.
	Thus the implicit function theorem can be applied to produce a $k$ times continuously differentiable (or analytic) function $g\colon U\rightarrow \RR_{++}$ for some open neighborhood $U$ of $y$ such that
	$$\graph \radTransSup{f} \cap (U\times \RR_{++}) = \{(y,g(y))\mid y\in U\}.$$
	As a result, $\radTransSup{f}$ must equal $g$ near $y$, and hence is also $k$ times continuously differentiable (or analytic) near $y$.
	
	All that remains is to derive our gradient formula and show it satisfies the claimed inequality. Consider any $y\in\dom f^\Gamma$ and set $x=y/f^\Gamma(y)\in\dom f$. The density theorem of proximal calculus~\cite[Theorem 1.3.1]{Clarke1998-nonsmoothanalysis} guarantees a sequence $y_i\rightarrow y$ exists with all $\partial_P \radTransSup{f}(y_i)\neq\emptyset$. Then using the subgradient formula of Proposition~\ref{prop:proximal-subgradients} and letting $x_i = y_i/\radTransSup{f}(y_i)$,
	$$\nabla \radTransSup{f}(y_i) = \frac{\nabla f(x_i)}{(\nabla f(x_i),-1)^T(x_i,f(x_i))}$$ 
	since $N^P_{\epi f}(x_i,f(x_i)) \subseteq \{\lambda(\nabla f(x_i),-1)\mid \lambda\geq 0\}$. Since $x_i\rightarrow x$, the continuous differentiability of $f$ and $f^\Gamma$ ensures
	$$ \nabla \radTransSup{f}(y) = \frac{\nabla f(x)}{(\nabla f(x),-1)^T(x,f(x))}.$$
	From this, it is immediate that
	\begin{align*}
	(\nabla \radTransSup{f}(y),-1)^T(y,\radTransSup{f}(y)) &= \left(\frac{\nabla f(x)}{(\nabla f(x),-1)^T(x,f(x))},-1\right)^T\left(\frac{x}{f(x)}, \frac{1}{f(x)}\right)\\
	&=\frac{1}{f(x)}\left(\frac{\nabla f(x)^Tx}{(\nabla f(x),-1)^T(x,f(x))} -1\right)\\
	&=\frac{1}{(\nabla f(x),-1)^T(x,f(x))}<0. \tag*{\qed}
	\end{align*}
\end{proof}
We remark that this result does not capture all functions for which the radial transformation is differentiable. For example, the strictly upper radial function
$$ f(x) = \begin{cases} 1+\sqrt{1-x^2} & \text{if }-1\leq x\leq 1\\ 0 &\text{otherwise}\end{cases}$$
is not differentiable everywhere in its domain (namely, it fails at $x=\pm1$). However, its upper radial transformation is differentiable everywhere
$$ \radTransSup{f}(y) = \begin{cases} (y^2+1)/2 &\text{if } -1\leq y\leq 1\\ |y| &\text{otherwise.}\end{cases}$$

Differentiating the gradient formula of Proposition~\ref{prop:differentiable-preserved} directly gives a Hessian formula for the radial transformation of a function. The discussion following Proposition 4.1 of~\cite{Artstein2017} derived a similar formula for convex functions. Therein further results are discussed on determinants of these objects, which can be readily computed since $\Gamma$'s determinant can be easily computed. 
\begin{proposition}\label{prop:hessian-conversion}
	Consider continuous, strictly radial $f$ with points $(x,f(x)) = \Gamma(y,\radTransSup{f}(y))$. If $f$ is twice continuously differentiable around $x$ with
	$$(\nabla f(x),-1)^T(x,f(x))<0,$$
	the Hessian of $\radTransSup{f}=\radTransInf{f}$ at $y$ is given by
	$$ \nabla^2\radTransSup{f}(y) = \frac{f(x)}{(\nabla f(x),-1)^T(x,f(x))}\cdot J\nabla^2 f(x)J^T$$
	where $J = I -\frac{\nabla f(x)x^T}{(\nabla f(x),-1)^T(x,f(x))}$.
\end{proposition}
\begin{proof}
	Denote the bijection relating the domains of $\radTransSup{f}$ and $f$ by $\pi(y) = y/\radTransSup{f}(y)$ (as shown by Lemma~\ref{lem:graph-bijection}). Then the gradient of the radial transformation is
	$$ \nabla \radTransSup{f}(y) = \frac{\nabla f(\pi(y))}{(\nabla f(\pi(y)),-1)^T(\pi(y),f(\pi(y)))}.$$
	Thus the Jacobian of $\pi$ is given by
	\begin{align*}
	\nabla \pi(y) &= \frac{I}{\radTransSup{f}(y)} - \frac{y\nabla \radTransSup{f}(y)^T}{\radTransSup{f}(y)^2}\\
	&=\frac{1}{\radTransSup{f}(y)}\left( I - \frac{y\nabla f(\pi(y))^T}{\radTransSup{f}(y)(\nabla f(\pi(y)),-1)^T(\pi(y),f(\pi(y)))}\right)\\
	&=f(\pi(y))\left( I - \frac{\pi(y)\nabla f(\pi(y))^T}{(\nabla f(\pi(y)),-1)^T(\pi(y),f(\pi(y)))}\right)
	\end{align*}
	where the third equality uses that $y/\radTransSup{f}(y) = \pi(y)$ and $1/\radTransSup{f}(y) = f(\pi(y))$ by Lemma~\ref{lem:graph-bijection}. Let $g(x) = \nabla f(x)/(\nabla f(x),-1)^T(x,f(x))$. Noting that the gradient of $(\nabla f(x),-1)^T(x,f(x))$ is $\nabla^2 f(x)^Tx$, the Jacobian of $g$ is given by
	\begin{align*}
	\nabla g(x) 
	&= \frac{\nabla^2 f(x)}{(\nabla f(x),-1)^T(x,f(x))} - \frac{\nabla f(x)x^T\nabla f^2(x)}{(\nabla f(x),-1)^T(x,f(x))^2}\\
	&= \frac{1}{(\nabla f(x),-1)^T(x,f(x))}\left(I - \frac{\nabla f(x) x^T}{(\nabla f(x),-1)^T(x,f(x))}\right)\nabla^2 f(x).
	\end{align*}
	Since $\nabla \radTransSup{f}(y) = g(\pi(y))$, the Hessian of $\radTransSup{f}$ is given by $\nabla g(\pi(y)) \nabla \pi(y)$ which is exactly the claimed formula. \qed
\end{proof}

\subsection{Optimality Under the Radial Transformation}\label{subsec:optimal}
Before addressing optimality under our radial duality, we observe that inequalities between functions are reversed by applying either radial function transformation. This mirrors~\eqref{eq:radial-inclusion}, where we saw the radial set transformation preserves inclusions between sets. We say $f \leq g$ if $f(x)\leq g(x)$ for all $x\in\varSpace$.
\begin{lemma}\label{lem:order-preserving}
	For any functions $f,g$, if $f\leq g$, then $\radTransSup{g}\leq\radTransSup{f}$ and $\radTransInf{g}\leq\radTransInf{f}$.
\end{lemma}
\begin{proof}
	Notice that $f\leq g$ is equivalent to $\epi g \subseteq \epi f$. Then~\eqref{eq:radial-inclusion} gives $\radTransSet{(\epi g)} \subseteq \radTransSet{(\epi f)}$. Therefore $\radTransSup{f}(y)\geq\radTransSup{g}(y)$ for all $y\in\varSpace$. \qed
\end{proof}

Now we consider how the extreme values and points of a function and its radial transformations relate. First, we show for radial functions, the supremum value of $f$ equals the reciprocal of the infimum value of $\radTransSup{f}$ in Proposition~\ref{prop:optimality-characterization}. Then Proposition~\ref{prop:minimizer-characterization} shows the maximizers of $f$ are
related to minimizers of $\radTransSup{f}$ by the radial point transformation.

\begin{proposition}\label{prop:optimality-characterization}
	For any function $f$, $\left(\inf f\right)\cdot \left(\sup \radTransSup{f}\right) = 1$ where we let $\infty\cdot 0 = 0\cdot\infty = 1$. Further, if $f$ is upper radial, $\left(\sup f\right)\cdot \left(\inf \radTransSup{f}\right) = 1.$\\
	Likewise, $\left(\sup f\right)\cdot \left(\inf \radTransInf{f}\right) = 1$ and if $f$ is lower radial, $\left(\inf f\right)\cdot \left(\sup \radTransInf{f}\right) = 1.$
\end{proposition}
\begin{proof}
	Observe that $f \geq \inf f$ and so Lemma~\ref{lem:order-preserving} implies
	$ \radTransSup{f} \leq \radTransSup{\left(\inf f\right)} = 1/\inf f.$
	First we show the $\geq$ inequality (which is trivial if $\inf f=\infty$). Suppose $\inf f<\infty$.
	Let $x_i$ be a sequence with $\lim f(x_i) = \inf f$. Then fix any $\epsilon>0$ and set $y_i = x_i/(f(x_i)+\epsilon)$. Observe that $v\cdot f(y_i/v)<1$ when $v=1/(f(x_i)+\epsilon)$. Therefore $\radTransSup{f}(y_i) \geq 1/(f(x_i)+\epsilon)$ and so taking the limit as $\epsilon\rightarrow 0$ implies
	$ \sup \radTransSup{f} \geq 1/\inf f.$
	Lastly, supposing $f$ is upper radial, Theorem~\ref{thm:radial-function-duality} ensures
	$1 = \left(\inf \radTransSup{f}\right)\cdot \left(\sup \biradTransSup{f}\right) = \left(\inf \radTransSup{f}\right)\cdot \left(\sup f\right).$ \qed
\end{proof}
\begin{proposition}\label{prop:minimizer-characterization}
	For any upper radial $f$ with $\sup f\in\RR_{++}$,
	$$ (\argmin \radTransSup{f} )\times\{\inf \radTransSup{f}\} \subseteq \radTransSet{\left((\argmax f )\times\{\sup f\}\right)}.$$
	Likewise for any lower radial $f$ with $\inf f\in\RR_{++}$,
	$$ (\argmax \radTransInf{f} )\times\{\sup \radTransInf{f}\} \subseteq \radTransSet{\left((\argmin f )\times\{\inf f\}\right)}.$$
	If $f^p(y,\cdot)$ is strictly increasing on its domain, equality holds in both cases.
\end{proposition}
\begin{proof}
	Consider any $(y,v)\in (\argmin \radTransSup{f} )\times\{\inf \radTransSup{f}\}$ and set $(x,u)=\Gamma(y,v)$. Then $(y,v)\in\epi\radTransSup{f}$, and so Lemma~\ref{lem:hypo-epi-inclusion} ensures $(x,u)\in \hypo f$. Therefore $x$ attains the maximum value of $f$ by Proposition~\ref{prop:optimality-characterization}. Hence $(x,u)\in(\argmax f )\times\{\sup f\}$.
	When $f^p(y,\cdot)$ is strictly increasing, equality holds in Lemma~\ref{lem:hypo-epi-inclusion} and the above argument can be repeated in reverse. \qed
\end{proof}

For nonconvex optimization problems, finding global solutions is often intractable and so the focus of many optimization methods is on finding stationary points (that is, points with a zero sub(sup)gradient in their sub-(sup-)differential). Just as optimal solutions were related between the primal and radial dual problems, stationary points are also directly related by the radial point transformation.
\begin{proposition}\label{prop:stationary-characterization}
	For any strictly upper radial $f$,
	$$ \{(y,\radTransSup{f}(y))\in\varSpace\times\RR_{++} \mid 0\in\partial_P\radTransSup{f}(y)\} = \radTransSet{\{(x,f(x))\in\varSpace\times\RR_{++} \mid 0\in\partial^Pf(x)\}}.$$
	Likewise for any strictly lower radial $f$,
	$$ \{(y,\radTransInf{f}(y))\in\varSpace\times\RR_{++} \mid 0\in\partial^P\radTransInf{f}(y)\} = \radTransSet{\{(x,f(x))\in\varSpace\times\RR_{++} \mid 0\in\partial_Pf(x)\}}.$$
\end{proposition}
\begin{proof}
	Suppose $0\in\partial^Pf(x)$. Then $(0,1)\in N^P_{\hypo f}((x,f(x)))$ and $(0,1)^T(x,f(x))=f(x)>0$. Hence by Proposition~\ref{prop:proximal-subgradients}, $y=x/f(x)$ is a stationary point of the radial dual as $0/f(x)=0\in\partial_Pf^\Gamma(y)$. Supposing instead $0\in \partial_Pf^\Gamma(y)$, Proposition~\ref{prop:proximal-subgradients} ensures some $(\zeta,\delta)\in N^P_{\hypo f}((x,f(x)))$ has $\zeta/(\zeta,\delta)^T(x,f(x)) = 0$ and $(\zeta,\delta)^T(x,f(x))>0$. Thus $\zeta=0,\delta>0$ and so $0=-\zeta/\delta\in\partial^P f(x)$. \qed
\end{proof}

\section{Characterizing Epigraph Reshaping Transformations}\label{sec:axioms}

In this final section, we consider a broader class of transformations given by reshaping a function's epigraph via some mapping $G$. In particular, given a generic optimization problem
\begin{equation}\label{eq:maximization}
p^* = \max_{x\in\varSpace} f(x),
\end{equation}
we transform $f$ by reshaping its epigraph into the hypograph of a new function
$$ f^G(y) = \sup\{v \mid (y,v) \in G(\epi f)\}.$$
If $G(\epi f)$ is indeed a function's hypograph, the identity $\hypo f^G = G(\epi f)$ holds.
Then the transformed optimization problem is defined as
\begin{equation}\label{eq:minimization}
d^* = \min_{y\in\varSpace} f^G(y).
\end{equation}

Paralleling the development of the radial function transformation $f^\Gamma$, we would like to relate minimizers of $f^G$ to maximizers of $f$ and vice versa through the mapping $G$. To this end, we assume {\bf invertibility}:
\begin{equation}\label{eq:affine-bijection}\tag{A1}
G \text{ is a bijection.}
\end{equation}
Whenever the given function $f$ is concave, we want to preserve this structure by having $f^G$ be convex. To ensure this, we assume $G$ is {\bf convexity preserving}: for any $S\subseteq \dom G$,
\begin{equation}\label{eq:affine-intervals}\tag{A2}
S \text{ is convex} \implies GS\text{ is convex.} 
\end{equation}
Lastly, we need a relationship between minimizers of $f$ and maximizers of $f^G$. This follows by assuming $G$ is {\bf height reversing}: for any pairs $(x,u)=G^{-1}(y,v)\in\dom G$ and $(x',u')=G^{-1}(y',v')\in\dom G$,
\begin{equation}\label{eq:affine-ordering}\tag{A3}
u\geq u' \implies v\leq v'.
\end{equation}

Under these three assumptions, any function satisfying $\hypo f^G = G(\epi f)$ must have $\argmax f \times \{p^*\} = G^{-1}(\argmin f^G\times\{d^*\})$. Hence the problems~\eqref{eq:maximization} and~\eqref{eq:minimization} are equivalent and converting points between these problems only requires evaluating $G$ or its inverse.

First as an example, we consider transforming functions $f\colon \varSpace \rightarrow \RR\cup\{\pm\infty\}$ mapping into the extended reals. Then $G$ must map $\varSpace\times\RR$ into $\varSpace\times\RR$.
We may additionally want to impose a condition requiring the function transformation $f^G$ satisfies $\hypo f^G = G(\epi f)$ for a reasonably large class a functions. Namely, we assume the function transformation is {\bf well-defined}: for all linear $f\colon \varSpace\rightarrow\RR$,
\begin{equation}\label{eq:affine-wellness}\tag{A4}
\hypo f^G = G(\epi f).
\end{equation}

The Fundamental Theorem of Affine Geometry (stated below) gives us an immediate way to characterize what possible transformations satisfy these four assumptions. See~\cite{Artin1988,Prasolov2001} as references.
\begin{theorem}\label{thm:affine-transformation}
	For $n\geq 2$, if $F \colon \RR^n \rightarrow \RR^n$ is a bijective, convexity preserving map, then $F$ is an affine transformation.
\end{theorem}
From this, one can conclude any duality of this form amounts to the trivial duality between maximizing a function and minimizing its negative. Thus there are notable limitations on what a transformation satisfying (A1)-(A4) can accomplish. However, the following theorem provides an alternative to using the Fundamental Theorem of Affine Geometry and facilitates studying more general transformations. See~\cite{Artstein2012} for a reference or \cite{Shiffman1995} for the original version of this result, which takes a more general perspective based in projective spaces.
\begin{theorem}\label{thm:fractional-transformation}
	For $n\geq 2$, if for some convex set $K\subseteq\RR^n$ with nonempty interior, $F \colon K \rightarrow \RR^n$ is an injective, convexity preserving map, then $F$ is a fractional linear map.
\end{theorem}

This indicates that there is more potential for interesting transformations if we can restrict our assumptions to a convex subset of $\varSpace\times\RR$ (in our case, $\varSpace \times \RR_{++}$).
To this end, we now consider transforming functions $f\colon \varSpace\rightarrow \extPos$ mapping into the extended positive reals.

We also suppose the transformed function $f^G\colon \varSpace\rightarrow \extPos$ maps into the extend positive reals, and so $G$ maps $\varSpace \times \RR_{++}$ into $\varSpace \times \RR_{++}$. Our first three assumptions (namely, invertibility~\eqref{eq:affine-bijection}, convexity preserving~\eqref{eq:affine-intervals}, and height reversing~\eqref{eq:affine-ordering}) extend directly to $G$ having restricted domain and codomain. We consider the following assumption paralleling~\eqref{eq:affine-wellness} to ensure the function transformation is {\bf well-defined} for a basic class of linear-like functions: for all linear $f\colon \varSpace\rightarrow\RR$,
\begin{equation}\label{eq:fractional-wellness}\tag{B4}
\hypo (f_+)^G = G(\epi f_+)
\end{equation}
where $(\cdot)_+ = \max\{\cdot,0\}$ denotes nonnegative thresholding.
Under these four assumptions, we find that any such transformation of nonnegative-valued optimization problems must produce an affinely shifted version of the upper radial function transformation. This is proven in Section~\ref{subsec:fractional-transformation}.
\begin{theorem}\label{thm:fractional-transformation-functions}
	Consider any map $G\colon \varSpace\times\RR_{++} \rightarrow\varSpace\times\RR_{++}$ satisfying~\eqref{eq:affine-bijection}, \eqref{eq:affine-intervals}, and~\eqref{eq:affine-ordering}. Then $G$ is a fractional linear map
	$ G(x,u) = (Ax+\alpha u + b, d)/u $
	with $d>0$. Furthermore, if~\eqref{eq:fractional-wellness} holds, then $b=0$ and the function transformation is given by
	$ f^G(y) = d f^{\Gamma}(A^{-1}(y-\alpha)).$
\end{theorem}
This result is similar in spirit to~\cite[Theorem 5]{Artstein2011}, avoiding their reliance on nonnegative convex functions with value $0$ at the origin.

\subsection{Proof of Theorem~\ref{thm:fractional-transformation-functions}} \label{subsec:fractional-transformation}
From assumptions~\eqref{eq:affine-bijection} and~\eqref{eq:affine-intervals}, Theorem~\ref{thm:fractional-transformation} immediately implies that 
$$ G(x,u) = \frac{\begin{bmatrix} A & \alpha \\ \beta^T & c
	\end{bmatrix} \begin{bmatrix} x \\ u \end{bmatrix} + \begin{bmatrix} b \\ d \end{bmatrix}}{\eta^Tx +gu +h}.$$
Since $G$ is a bijection, all $(x,u)\in\varSpace\times\RR_{++}$ have $G(x,u)\in\varSpace\times\RR_{++}$, and so
$$ \frac{\beta^Tx + cu +d}{\eta^Tx + gu +h} >0.$$
Thus $\beta = \eta =0$. Then the mapping $\sigma(u) = (cu +d)/(gu +h)$ must be a bijection from $\RR_{++}$ to $\RR_{++}$. Therefore $\sigma$ must be one of the following forms: either $g=0$ and so $\sigma(u) = cu/h$ with $c/h>0$, or $g\neq 0$ and so $\sigma(u) = d/(gu)$ with $d/g>0$. 
From~\eqref{eq:affine-ordering}, the latter of these two possibilities must be the case. Thus $g\neq0$ and so without loss of generality, we can suppose $g=1$ and $d>0$. Hence,
$ G(x,u) = (Ax+\alpha u+b, d)/u.$

Now additionally assume that this transformation is well-defined for all linear functions (after thresholding to nonnegative values), namely~\eqref{eq:fractional-wellness}. Observe that the inverse of $G$ is given by
$ G^{-1}(y,v) = d(A^{-1}(y - (b/d)v - \alpha),1)/v.$
Consider the linear function $f(x) = (b^TAx)_+$, which has $G(\epi f) $ equal to
\begin{align*} 
&\{(y,v)\in\varSpace\times\RR_{++} \mid G^{-1}(y,v)\in\epi f\}\\
&=\{(y,v)\in\varSpace\times\RR_{++} \mid (b^TA(dA^{-1}(y-(b/d)v - \alpha))/v)_+ \leq d/v\}\\
&=\{(y,v)\in\varSpace\times\RR_{++} \mid b^T(y-(b/d)v - \alpha) \leq 1\}\\
&=\{(y,v)\in\varSpace\times\RR_{++} \mid \|b\|^2_2v \geq db^T(y-\alpha) - d\}.
\end{align*}
For any $b\neq 0$, this not the hypograph of any function, contradicting~\eqref{eq:fractional-wellness}. Thus we must have $b=0$. From this, we have $G^{-1}(y,v) = d(A^{-1}(y-\alpha), 1)/v$.
Therefore $f^G$ must be an affine translation of $f^\Gamma$ since
\begin{align*}
f^G(y) &= \sup\{v>0 \mid G^{-1}(y,v)\in\epi f\}\\
&= \sup\{v>0 \mid f(dA^{-1}(y-\alpha)/v) \leq d/v\}\\
&= d\sup\{w>0 \mid f(A^{-1}(y-\alpha)/w) \leq 1/w\}\\
&= d f^{\Gamma}(A^{-1}(y-\alpha)).
\end{align*}

\begin{acknowledgements}
	The author thanks Jim Renegar broadly for inspiring this work and concretely for providing feedback on multiple drafts. Further, Jim had the valuable idea to use the implicit function theorem to simplify and improve the proof of Proposition~\ref{prop:differentiable-preserved}. Additionally, three anonymous referees and the associate editor provided useful feedback much improving this work's presentation and clarity.
\end{acknowledgements}

%
%

\bibliographystyle{spmpsci}      
\bibliography{references}   

\appendix
\section{Proofs Computing Some Radial Set Transformations}
\allowdisplaybreaks

\subsection{Proof of Proposition~\ref{prop:convex-set-radial}}
	It suffices to show $S$ being convex implies $\radTransSet{S}$ is convex, since the duality of the radial set transformation~\eqref{eq:radial-set-duality} will then imply the reverse direction. Consider any $(y,v),(y',v')\in\radTransSet{S}$. Let $(x,u) = \Gamma(y,v)$ and $(x',u')=\Gamma(y',v')$. Then
	$\lambda(x,u) + (1-\lambda)(x',u') \in S$ for any $0\leq\lambda\leq 1$.
	Therefore the line segment between $(y,v)$ and $(y',v')$ lies in $\radTransSet{S}$ as
	\begin{align*}
	\radTransSet{S}&\ni \frac{(\lambda x + (1-\lambda)x', 1)}{\lambda u + (1-\lambda)u'}
	=\frac{\lambda/v}{\lambda/v + (1-\lambda)/v'}(y,v) + \frac{(1-\lambda)/v'}{\lambda/v + (1-\lambda)/v'}(y',v').
	\end{align*}

\subsection{Proof of Proposition~\ref{prop:halfspace-radial}}
	It suffices to show $S$ being a halfspace implies $\radTransSet{S}$ is a halfspace, since the duality of the radial set transformation~\eqref{eq:radial-set-duality} will then imply the reverse direction. By definition, we have
	\begin{align*}
	\radTransSet{S}&=\left\{\frac{(x',1)}{u'}\in\varSpace\times\RR_{++} \mid \begin{bmatrix} \zeta \\ \delta 
	\end{bmatrix}^T\begin{bmatrix} x'-x \\ u'-u 
	\end{bmatrix}\leq 0\right\}\\
	&= \left\{(y',v')\in\varSpace\times\RR_{++} \mid  \begin{bmatrix} \zeta \\ \delta 
	\end{bmatrix}^T\begin{bmatrix} y'/v'-x \\ 1/v'-u 
	\end{bmatrix}\leq 0 \right\}\\
	&= \left\{(y',v')\in\varSpace\times\RR_{++} \mid  \begin{bmatrix} \zeta \\ \delta 
	\end{bmatrix}^T\begin{bmatrix} y'-v'x \\ 1-v'u 
	\end{bmatrix}\leq 0 \right\}\\
	&= \left\{(y',v')\in\varSpace\times\RR_{++} \mid  \begin{bmatrix} \zeta \\ -(\zeta,\delta)^T(x,u)
	\end{bmatrix}^T\begin{bmatrix} y' \\ v' 
	\end{bmatrix} +\delta\leq 0 \right\}\\
	&= \left\{(y',v')\in\varSpace\times\RR_{++} \mid  \begin{bmatrix} \zeta \\ -(\zeta,\delta)^T(x,u)
	\end{bmatrix}^T\begin{bmatrix} y'-y \\ v'-v 
	\end{bmatrix} \leq 0 \right\}.
	\end{align*}

\subsection{Proof of Proposition~\ref{prop:ellipsoid-radial}}
	It suffices to show $S$ being an ellipsoid in $\varSpace\times\RR_{++}$ implies $\radTransSet{S}$ is an ellipsoid $\varSpace\times\RR_{++}$, since the duality of the radial set transformation~\eqref{eq:radial-set-duality} will then imply the reverse direction. Denote the blocks of $H$ by $\begin{bmatrix} H_{11} & H_{12}\\ H_{12}^T & H_{22} \end{bmatrix}$ and define the following matrix
	$$ G = \begin{bmatrix} H_{11} & -\begin{bmatrix} H_{11} \\ H^T_{12}
	\end{bmatrix}^T\begin{bmatrix} x \\ u\end{bmatrix}\\ -\begin{bmatrix} x \\ u\end{bmatrix}^T\begin{bmatrix} H_{11} \\ H^T_{12}
	\end{bmatrix} & \begin{bmatrix} x \\ u 
	\end{bmatrix}^T \begin{bmatrix} H_{11} & H_{12}\\ H_{12}^T & H_{22} 
	\end{bmatrix} \begin{bmatrix} x \\ u 
	\end{bmatrix} -1
	\end{bmatrix}, $$
	related to the radially dual ellipsoid.	
	For any ellipsoid $S\subseteq\varSpace\times\RR_{++}$ defined by \eqref{eq:basic-ellipsoid}, we claim that $\radTransSet{S}$ is the following ellipsoid in $\varSpace\times\RR_{++}$ with center $\begin{bmatrix}y \\ v\end{bmatrix} = G^{-1}\begin{bmatrix}-H_{12}\\ H_{12}^Tx + H_{22}u \end{bmatrix}$
	$$\radTransSet{S} = \left\{(y',v')\in\varSpace\times\RR_{++} \mid \begin{bmatrix} y'-y \\ v'-v \end{bmatrix}^T \left(\frac{G}{\begin{bmatrix}y \\ v\end{bmatrix}^T G \begin{bmatrix}y \\ v\end{bmatrix} - H_{22}}\right) \begin{bmatrix} y'-y \\ v'-v \end{bmatrix} \leq 1\right\}.$$
	
	First, we observe that $G$ is indeed positive definite. Since $H$ is positive definite, considering its Schur complements ensures $H_{11}$ is positive definite and $H_{22} - H_{12}^TH_{11}^{-1}H_{12}>0$. Likewise, $G$ is positive definite if $H_{11}$ is positive definite and
	$$\left(\begin{bmatrix} x \\ u 
	\end{bmatrix}^T \begin{bmatrix} H_{11} & H_{12}\\ H_{12}^T & H_{22} 
	\end{bmatrix} \begin{bmatrix} x \\ u 
	\end{bmatrix} -1\right) - \left(\begin{bmatrix} H_{11} \\ H^T_{12}
	\end{bmatrix}^T\begin{bmatrix} x \\ u\end{bmatrix}\right)^TH_{11}^{-1}\left(\begin{bmatrix} H_{11} \\ H^T_{12}
	\end{bmatrix}^T\begin{bmatrix} x \\ u\end{bmatrix}\right)> 0.$$
	Simplifying this condition for  $G$ to be positive definite yields the equivalent inequality
	$u^2(H_{22} - H_{12}^TH_{11}^{-1}H_{12}) >1$.
	{\color{blue} To see why this is the case, recall $S\subseteq \varSpace\times\RR_{++}$. Computing the minimum height of a point in $S$ gives $\min_{(\bar x,\bar u)\in S} \bar u =  u - 1/\sqrt{H_{22} - H_{12}^TH_{11}^{-1}H_{12}} >0$.} 
	Applying $\Gamma$ to $S$ and then completing the square in the final line gives the claim as
	\begin{align*}
	\radTransSet{S} &= \left\{\frac{(x',1)}{u'}\in\varSpace\times\RR_{++} \mid \begin{bmatrix} x'-x \\ u'-u 
	\end{bmatrix}^T \begin{bmatrix} H_{11} & H_{12}\\ H_{12}^T & H_{22} 
	\end{bmatrix} \begin{bmatrix} x'-x \\ u'-u 
	\end{bmatrix} \leq 1\right\}\\
	&= \left\{(y',v')\in\varSpace\times\RR_{++} \mid \begin{bmatrix} y'/v'-x \\ 1/v'-u 
	\end{bmatrix}^T \begin{bmatrix} H_{11} & H_{12}\\ H_{12}^T & H_{22} 
	\end{bmatrix} \begin{bmatrix} y'/v'-x \\ 1/v'-u 
	\end{bmatrix} \leq 1\right\}\\
	&= \left\{(y',v')\in\varSpace\times\RR_{++} \mid \begin{bmatrix} y'-v'x \\ 1-v'u 
	\end{bmatrix}^T \begin{bmatrix} H_{11} & H_{12}\\ H_{12}^T & H_{22} 
	\end{bmatrix} \begin{bmatrix} y'-v'x \\ 1-v'u 
	\end{bmatrix} \leq v'^2\right\}\\
	&= \left\{(y',v')\in\varSpace\times\RR_{++} \mid \begin{bmatrix} y' \\ v' 
	\end{bmatrix}^T G \begin{bmatrix} y' \\ v' 
	\end{bmatrix} +2H_{12}^Ty'-2(H_{12}^Tx+H_{22}u)v'\leq -H_{22}\right\}.
	\end{align*}

\end{document}